\documentclass[a4paper,11pt,headings=small,abstract=true]{scrartcl}
\usepackage[linktocpage, colorlinks]{hyperref} 
	\usepackage{color}
 	\definecolor{darkred}{rgb}{0.5,0,0}
 	\definecolor{darkgreen}{rgb}{0,0.5,0}
 	\definecolor{darkblue}{rgb}{0,0,0.5} 	\hypersetup{colorlinks,linkcolor=darkblue,filecolor=darkgreen,urlcolor=darkred,citecolor=darkblue}
\hypersetup{
pdftitle={A Logvinenko-Sereda theorem for vector-valued functions and application to control theory},
pdfauthor={Clemens Bombach, Martin Tautenhahn},
}
\usepackage[sc]{mathpazo}    
\usepackage[scaled]{helvet}  
\usepackage{eulervm}		
\usepackage[utf8]{inputenc}
\usepackage[T1]{fontenc}
\usepackage[english]{babel}
\usepackage{enumerate}
\usepackage{amssymb,amsfonts,amsmath,amsthm}
\usepackage{mathabx} 
\usepackage{color}
\usepackage[noblocks]{authblk}

\numberwithin{equation}{section}
\DeclareMathOperator{\real}{Re}
\DeclareMathOperator{\imag}{Im}
\DeclareMathOperator{\spt}{supp}
\DeclareMathOperator{\dom}{dom}
\newcommand{\ii}{\mathrm{i}}
\newcommand{\e}{\mathbf{e}}
\newcommand{\euler}{\mathrm{e}}
\renewcommand{\L}{\mathcal{L}}
\newcommand{\T}{\mathbb{T}}
\newcommand{\D}{\mathcal{D}}
\newcommand{\Schw}{\mathcal{S}}
\newcommand{\F}{\mathcal{F}}
\newcommand{\thickset}{E}
\newcommand{\ball}[2]{B (#1 , #2)}
\newcommand{\ballc}[1]{B (#1)}
\newcommand{\disc}[2]{D (#1 , #2)}
\newcommand{\discc}[1]{D (#1)}
\newcommand{\Eins}{\mathbf{1}}
\newcommand{\Z}{\mathbb{Z}}
\newcommand{\C}{\mathbb{C}}
\newcommand{\R}{\mathbb{R}}
\newcommand{\N}{\mathbb{N}}
\newcommand{\diff}{\mathrm{d}}
\newcommand{\abs}[1]{{\lvert #1 \rvert}}
\newcommand{\dual}[1]{\langle#1\rangle}
\newcommand{\loc}{\mathrm{loc}}
\newcommand{\ran}{\mathrm{ran}}
\newtheorem{theorem}{Theorem}[section]
\newtheorem{lemma}[theorem]{Lemma}
\newtheorem{proposition}[theorem]{Proposition}
\theoremstyle{definition}

\newtheorem{example}[theorem]{Example}
\theoremstyle{remark}
\newtheorem{remark}[theorem]{Remark}
\title{A Logvinenko-Sereda theorem for vector-valued functions and application to control theory}
\author{Clemens Bombach}
\affil{Technische Universit\"at Chemnitz, Fakult\"at f\"ur Mathematik, 09126 Chemnitz, Germany}
\author{Martin Tautenhahn}
\affil{Universit\"at Leipzig, Fakult\"at f\"ur Mathematik und Informatik, 04109 Leipzig, Germany}
\date{\vspace{-4ex}}
\begin{document}
\maketitle
\begin{abstract}
We prove a Logvinenko-Sereda Theorem for vector valued functions. That is, for an arbitrary Banach space $X$, all $p \in [1,\infty]$, all $\lambda \in (0,\infty)^d$, all $f \in L^p (\R^d ; X)$ with $\spt \F f \in \times_{i=1}^d (-\lambda_i/2 , \lambda_i /2)$, and all thick sets $E \subseteq \R^d$ we have
\begin{equation*}
\lVert \Eins_E f \rVert_{L^p (\R^d;X)} \geq 
C  \lVert f \rVert_{L^p (\R^d;X)} .
\end{equation*}
The constant is explicitly known in dependence of the geometric parameters of the thick set and the parameter $\lambda$.
As an application, we study control theory for normally elliptic operators on Banach spaces whose coefficients of their symbol are given by bounded linear operators. This includes systems of coupled parabolic equations or problems depending on a parameter.
\\[1ex]
\\[1ex]
\end{abstract}
\section{Introduction} \label{sec:intro}
The paper is split into two parts. The first part concerns a generalization of the classical Logvinenko-Sereda Theorem to vector valued functions. The second part then studies an  application to control theory.
\par
The Logvinenko-Sereda Theorem goes back at least to the papers \cite{Panejah-61,Panejah-62}, and has been proven independently in \cite{Kacnelson-73,LogvinenkoS-74}. In order to formulate its result we introduce some notation.Let $\rho\in (0,1]$ and $L = (L_i)_{i=1}^d \in (0,\infty)^d$. A set $\thickset \subseteq \R^d$ is called \emph{$(\rho,L)$-thick} if $\thickset$ is measurable and for all $x \in \R^d$ we have
   \begin{equation} \label{eq:thick-set}
   \left\lvert \thickset \cap \left( \bigtimes_{i=1}^d (0,L_i) + x \right) \right\rvert \geq \rho \prod_{i=1}^{d} L_i.
  \end{equation}
Here, $\lvert \cdot \rvert$ denotes the Lebesgue measure. For $\lambda \in (0,\infty)^d$ we use the notation 
\begin{equation}\label{eq:parallelepiped}
\Pi_\lambda = \times_{i=1}^d (-\lambda_i/2 , \lambda_i / 2) 
\end{equation}
for the parallelepiped with side lengths $\lambda_i$, $i \in \{1,2,\ldots , d\}$. For $f \in L^p (\R^d)$ we denote by $\F f$ its Fourier transform.
The results of the above mentioned papers can be summarized as follows.
 \begin{theorem}\label{thm:LS-C}
For all $p \in [1,\infty]$, all $\lambda \in (0,\infty)^d$, all all $\rho > 0$, all $L \in (0,\infty)^d$, and all $(\rho , L)$-thick sets $\thickset \subseteq \R^d$ there exists a constant $C \geq 1$ such that for all $f \in L^p (\R^d)$ with $\spt \mathcal{F} f \subseteq \Pi_\lambda$ we have 
\begin{equation}\label{eq:LS-C}
\lVert \Eins_E f \rVert_{L^p (\R^d)} \geq 
C  \lVert f \rVert_{L^p (\R^d)} .
\end{equation}
\end{theorem}
Thus the result compares the overall $L^p$-norm of the function $f$ with its norm only on a thick subset $\thickset \subset \R^d$. 
The papers \cite{Kacnelson-73,LogvinenkoS-74} also show that the constant $C$ can be chosen as $C =  c_1 \euler^{c_2 \abs{\lambda}}$ with some positive constants $c_1$ and $c_2$ depending only on the space dimension and the geometric parameters $\rho$ and $L$.  This result has been significantly improved in \cite{Kovrijkine-00,Kovrijkine-01}, in which it is shown that $C$ can be chosen as 
\[
 C = \left( \frac{\rho}{K} \right)^{K (1 + \lambda \cdot L)}
\]
with some positive constant $K$ depending only on the dimension, which appears to be optimal. Subsequently, the classical Logvinenko-Sereda theorem has been adapted to various settings, e.g., to $L^2$-functions whose Fourier–Bessel transform is supported in an interval \cite{GhobberJ-12}, or to functions on the torus in \cite{EgidiV-20}.
\par
In the case $p =2$, the condition $\spt \mathcal{F}f \subseteq \Pi_\lambda$ is implied by $f \in \ran {P_{\sqrt{\lambda}}(-\Delta)}$ where $-\Delta$ denotes the negative Laplacian and $P_{\sqrt{\lambda}}(-\Delta)$ denotes the associated spectral projector on $L^2(\R^d)$ onto energies below $\sqrt{\lambda}$. One can therefore ask whether Theorem~\ref{thm:LS-C} continues to hold if we assume that $f \in \ran{ P_\lambda(H)}$ for a certain self-adjoint operator $H$ acting on $L^2(\R^d)$. This is indeed the case if $H = -\Delta_g  + V$ where $g$ is an analytic perturbation of the flat metric and $V:\R^d \to \R$ is analytic and decays at infinity, as shown in \cite{LebeauM}. Moreover, the recent \cite{EgidiS-21} provides a sufficient condition (a Bernstein-like inequality) for $f \in \ran{P_\lambda(H)}$ such that inequality \eqref{eq:LS-C} with $p = 2$ holds for thick observation sets $E$. Examples include the pure Laplacian, which is covered by Theorem~\ref{thm:LS-C}, divergence-type operators, and the harmonic oscillator.
\par
In this paper we generalize Theorem~\ref{thm:LS-C} to vector valued functions $f \in L^p (\R^d ; X)$ with values in an arbitrary Banach space $X$. It is formulated in Theorem~\ref{thm:LS}. Let us stress that the substantial novelty of Theorem~\ref{thm:LS} is that $X$ may be of infinite dimension. In particular, this allows to consider infinite dimensional state spaces in our application to control theory. This is the topic of the second part of our paper which we introduce in the following.
\par
We consider for $T > 0$ the linear control problem
\begin{equation}\label{eq:control-intro}
\partial_ty(t) + A_p y(t) = \mathbf{1}_Eu(t), \quad y(0) = y_0 \in \mathcal{X}^p = L^p (\R^d ; X), \quad t \in [0,T] ,
\end{equation}
where $X$ is an arbitrary Banach space, $p \in [1,\infty)$, $A_p$ is a normally elliptic differential operator in $\mathcal{X}^p$, and where $E \subset \R^d$ is a thick set. 
We study null-controllability in $L^r([0,T];\mathcal{X}^p)$ with $r \in [1,\infty]$, that is, for all $y_0 \in \mathcal{X}^p$ there exists a control function $u \in L^r([0,T];\mathcal{X}^p)$ such that the mild solution $y$ of \eqref{eq:control-intro} satisfies $y (T) = 0$. A weaker variant of this is approximate null-controllability. This means that for all $\varepsilon > 0$ and all $y_0 \in \mathcal{X}^p$ there exists a control function $u \in L^r([0,T];\mathcal{X}^p)$ such that the mild solution $y$ of \eqref{eq:control-intro} satisfies $\lVert y (T) \rVert < \varepsilon$. 
\par
Null-controllability for heat-like equations is well known in the scalar-valued case $X = \C$ and $p=r=2$, see, e.g., \cite{FattoriniR-71,LebeauR-95,FursikovI-96,ErvedozaZ-11,MartinRR-14} for bounded regions $\Omega \subset \R^d$, and \cite{Teresa-97, CabanillasMZ-01, MicuZ-01a, MicuZ-01, CannarsaMV-04, Miller-05, KalimerisO-20} for unbounded regions. 
We prove in Theorem~\ref{thm:control} that for arbitrary (possibly infinite-dimensional) Banach spaces $X$, the system \eqref{eq:control-intro} is approximately null-controllable if $p=1$ and null-controllable if $p \in (1,\infty)$. As a special case of our result one may consider, e.g., a system of $n$ coupled parabolic partial differential equations (if $X = \C^n$), or problems depending on a parameter (here $X$ is a function space). For example, our results apply to strongly elliptic control systems of the form
\begin{equation*}
	\partial_t y(t) + (-A\nabla \nabla^\top)^m y(t) + B y(t) = \Eins_E u(t), \quad y(0) = y_0 \in L^p(\R^d;\C^n), \quad t \in [0,T] ,
\end{equation*}
where $A, B \in \C^{n \times n}$ are such that $(-A\nabla \nabla^\top)^m$ is strongly elliptic. For a related result, we refer to \cite{Ammar-KhodjaBDG-09}, where controllability for finite dimensional systems is studied using a suitable Kalman rank condition. As another example we consider the following setting: For $\lambda \in [0,1]$, let 
\begin{equation*}
 	A_\lambda = \sum_{i,j=1}^d a_{i,j}(\lambda)\partial_i\partial_j
\end{equation*}
with $a_{i,j} \in C[0,1]$ and consider the parameter dependent linear control problem
\begin{equation}\label{eq:control-parameter}
\partial_tz(t) + A_\lambda z(t) = \mathbf{1}_Ev(t), \quad z(0) = z_{0} \in L^p (\R^d), \quad t \in [0,T] ,
\end{equation}
where we view $A_\lambda$ as an unbounded operator in $L^p(\R^d)$. Concerning the question of null-controllability, we remark that the control function $v$ may depend on the parameter $\lambda$. Next, we reformulate this as a single linear control problem in $L^p (\R^d ; C[0,1])$.
We write $\mathbf{a}_{i,j} \in \L(C[0,1])$ for the multiplication operator given by $f \mapsto a_{i,j} f$. Consider the operator
\begin{equation*}
	A =  \sum_{i,j=1}^d \mathbf{a}_{i,j}\partial_i\partial_j
\end{equation*}
acting on $L^p(\R^d;C[0,1])$. Under certain assumptions on the coefficients $a_{i,j}$, the operator $A$ is normally elliptic and Eq.~\eqref{eq:control-intro} with $A_p$ replaced by $A$ and $X = C[0,1]$ is well posed. Therefore, the parameter dependent Eq.~\eqref{eq:control-parameter} can be rewritten in the form \eqref{eq:control-intro} with, $X = C[0,1]$, $A_p = A$ and $y_0 = z_0 \otimes \mathbf{1}_{[0,1]}$. The thick set $E \subseteq \R^d$ in Eq.~\eqref{eq:control-intro} may be chosen as in Eq.~\eqref{eq:control-parameter}.
\par
For the proof of Theorem~\ref{thm:control} we employ the classical equivalence between (approximate) null-controllability and final state observability for the adjoint problem. This follows from Douglas' lemma, see \cite{Douglas-66} in the case of Hilbert spaces, and  \cite{Embry-73,DoleckiR-77,Harte-78,CurtainP-78,Carja-85,Carja-88,Forough-14} for its generalization to  Banach spaces. The observability estimate is formulated in Theorem~\ref{thm:obs}. Its proof is based on the classical Lebeau-Robbiano strategy. For Hilbert spaces it goes back to the papers \cite{LebeauR-95,LebeauZ-98,JerisonL-99,Miller-10} and was further studied, e.g., in \cite{TenenbaumT-11,WangZ-17,BeauchardP-18,NakicTTV-20,Barcena-PetiscoZ-21}. Recently it has been adapted to Banach spaces in \cite{GallaunST-20,BombachGST-23}. The main idea of this strategy is that a so-called spectral inequality and a dissipation estimate implies an observability estimate. While the spectral inequality is provided by our vector-valued version of the Logvinenko-Sereda theorem, the dissipation estimate is derived from representing the semigroup generated by $-A_p$ as a Fourier multiplier with an operator-valued symbol.
\section{Preliminaries}\label{sec:preliminaries}
The theory of vector-valued distributions was developed by Schwartz in \cite{Schwartz-57} and \cite{Schwartz-58}. In \cite{Amann-97}, this theory was applied to study vector valued Fourier multipliers. Further results in this direction can be found in \cite{Amann-19} and \cite{HytoenenNVW-16}. It turns out that we cannot literally apply these results for our purpose, we present in this section some basic properties of vector-valued distributions and Fourier multipliers.
\par
Let $X$ be a Banach space with norm $\lVert \cdot \rVert_X$. We denote by $\D(\R^d;X)$, $\Schw(\R^d;X)$ and $\mathcal{E}(\R^d;X)$ the spaces of $X$-valued test functions, Schwartz functions and smooth functions with the usual topologies, and by $\D^\prime(\R^d;X)$, $\Schw^\prime(\R^d;X)$ and $\mathcal{E}^\prime(\R^d;X)$ the spaces of $X$-valued distributions, tempered distributions and compactly supported distributions respectively. Note that $\mathcal{F}^\prime(\R^d;X) = \L(\mathcal{F}(\R^d) ;X)$ where $\mathcal{F} \in\{\D,\Schw,\mathcal{E}\}$. We denote by $\mathcal{O}_M(\R^d;X)$ the space of slowly increasing $X$-valued functions, that is $\varphi \in \mathcal{O}_M(\R^d;X)$ if for each multi-index $\alpha$ there exist constants $C_\alpha, m_\alpha$ such that
\begin{equation*}
	\lVert \partial^\alpha\varphi(x) \rVert_X \leq C_\alpha(1 + \lvert x \rvert)^{m_\alpha}, \quad (x \in \R^d) .
\end{equation*}
For $v \in X$ and $\varphi \in \D(\R^d)$, we denote by $\varphi \otimes v$ the element of $\D(\R^d;X)$ given by
\begin{equation*}
	(\varphi \otimes v)(x) = \varphi(x)v .
\end{equation*}
The set of these functions is called the set of \emph{elementary tensors}. The set of finite linear combinations of elementary tensors is dense in $\mathcal{F}(\R^d;X)$ where $\mathcal{F} \in \{\D,\D',\Schw,\Schw',\mathcal{E},\mathcal{E}',\allowbreak\mathcal{O}_M\}$.
\par
In the usual fashion, we may extend the operations of differentiation, multiplication by smooth functions and Fourier transform to the appropriate classes of distributions by duality. In the case of the Fourier transform, this can be done as follows. We define for $z,x  \in \C^d$ the Fourier character $\e_z(x) = \euler^{\ii z\cdot x}$. Note that $\e_z \in \mathcal{E}(\R^d)$ and that $z \mapsto \e_z(x)$ is entire. We define the Fourier transform $\F:\Schw(\R^d;X) \to \Schw(\R^d;X)$ by 
\begin{equation*}
	(\F\varphi)(\xi) = \int\limits_{\R^d} \e_{-\xi}\varphi \diff x .
\end{equation*}
It is an automorphism of $\Schw(\R^d;X)$ with inverse given by
\begin{equation*}
	(\F^{-1}\varphi)(x) = \frac{1}{(2\pi)^d}\int\limits_{\R^d} \e_{x}\varphi \diff \xi .
\end{equation*} 
If $f \in \Schw^\prime(\R^d;X)$, then we define the Fourier transform $\F:\Schw^\prime(\R^d;X) \to \Schw^\prime(\R^d;X)$ by
\begin{equation*}
	(\F f)(\varphi) = f(\F\varphi), \quad (\varphi \in \Schw(\R^d))
\end{equation*}
and obtain an automorphism of $\Schw^\prime(\R^d;X)$.
\par
If $u \in \mathcal{E}^\prime(\R^d;X)$, i.e. $u$ has compact support, then $\F^{-1}u \in \mathcal{E}(\R^d;X)$. Thus we may define the inverse Fourier-Laplace transform $\L:\mathcal{E}^\prime(\R^d;X) \to C^\infty(\C^d;X)$ by
\begin{equation*}
	(\L u)(z) = (\F^{-1}(\e_{\ii \imag z }u))(\real z) .
\end{equation*}
By checking that the Cauchy-Riemann differential equations hold for $\L u$, it follows that $\L u$ is an entire function. It follows that if $f \in \Schw^\prime(\R^d;X)$ is such that $\F f \in \mathcal{E}^\prime(\R^d;X)$, then $f$ can be extended to an entire function $f:\C^d \to X$ given by $\mathcal{L}\F f$. In particular $f$ is analytic on $\R^d$.

For $i = 0,1,2$ let $X_i$ be a Banach space with norm $\lVert \cdot \rVert_{X_i}$. By a \emph{multiplication} we mean a bilinear continuous map
\begin{equation*}
	\bullet: X_1 \times X_2 \to X_0, \quad (x_1,x_2) \mapsto x_1 \bullet x_2
\end{equation*}
such that 
\begin{equation*}
	\lVert x_1 \bullet x_2 \rVert_{X_0} \leq \lVert x_1 \rVert_{X_1}\lVert x_2 \rVert_{X_2} .
\end{equation*}
In particular, we will be interested in the cases where
\begin{enumerate}[(i)]
	\item $X_1 = \C$, $X_2 = X_0$ and $\lambda \bullet x = \lambda x$,
	\item $X_1 = X_2^\prime$, $X_0 = \C$ and $x^\prime \bullet x = \dual{x^\prime,x}$,
	\item $X_1 = \L(X_2,X_0)$ and $A \bullet x = Ax$.
\end{enumerate}
Note that the first two cases can be seen as special cases of the third case.
\par
From \cite{Amann-97}, we infer that any multiplication gives rise to a unique hypocontinuous bilinear map
\begin{equation*}
	B: \mathcal{E}(\R^d ; X_1) \times \D'(\R^d ;X_2) \to \D'(\R^d ;X_0), \quad f_1 \times f_2 \mapsto B(f_1,f_2)
\end{equation*}
such that for all $\varphi_1,\varphi_2 \in \mathcal{D}(\R^d)$, $x_1 \in X_1$, $x_2 \in X_2$ we have
\begin{equation*}
	B(\varphi_1 \otimes x_1,\varphi_2 \otimes x_2) = (\varphi_1 \varphi_2) \otimes (x_1 \bullet x_2) .
\end{equation*}
Here, hypocontinuous means that it is continuous in each variable, and uniformly continuous if one of the variables is restricted a bounded set. Furthermore, the restriction $B|_{\mathcal{O}_M(\R^d;X_1) \times \Schw^\prime(\R^d;X_2)}$ is hypocontinuous as well. We write $B(f_1,f_2) = f_1 \bullet f_2$ in the following.
\par
Set $D = - \ii \nabla$. Given $m \in \mathcal{O}_M(\R^d;X_1)$, we define the Fourier multiplier
\begin{equation*}
	m(D): \Schw'(\R^d;X_2) \to \Schw'(\R^d;X_0), \quad f \mapsto \F^{-1}(m \bullet \F f) .
\end{equation*}
We note that in the special case of $X_0 = X_2 = X$, $X_1 = \L(X)$, we have
\begin{equation*}
	m_1(D)m_2(D) = (m_1m_2)(D) .
\end{equation*}
With respect to the above multiplication $\bullet$, we define the convolution $*_\bullet$ of two elementary tensors $\varphi_1 \otimes x_1$ and $\varphi_2\otimes x_2$ (with $\varphi_1,\varphi_2 \in \mathcal{D}(\R^d)$, $x_1 \in X_1$, and $x_2 \in X_2$), by
\begin{equation*}
(\varphi_1 \otimes x_1) *_\bullet (\varphi_2 \otimes x_2) = (\varphi_1 * \varphi_2) \otimes (x_1 \bullet x_2) ,
\end{equation*}
where $*$ denotes the usual convolution of scalar-valued functions.
Theorem~3.1 in \cite{Amann-97} implies that $*_\bullet$ extends to bilinear, hypocontinuous maps:
\begin{align*}
	& *_\bullet: \Schw(\R^d;X_1) \times \Schw^\prime(\R^d;X_2) \to \Schw^\prime(\R^d;X_0), \\
	& *_\bullet: \mathcal{D}^\prime(\R^d;X_1) \times \mathcal{E}^\prime(\R^d;X_2) \to \mathcal{D}^\prime(\R^d;X_0) .
\end{align*}
Moreover, according to \cite[Theorem~3.5]{Amann-97}, for $1 \leq p \leq \infty$, there is a third extension
\begin{equation*}
*_\bullet:L^1(\R^d;X_1) \times L^p(\R^d;X_2) \to L^p(\R^d;X_0), \quad (f,g) \mapsto \int\limits_{\R^d} f(\cdot - y)\bullet g(y) \diff y ,
\end{equation*}
satisfying Young's inequality
\begin{equation*}
	\lVert f *_\bullet g \rVert_{L^p(\R^d;X_0)} \leq \lVert f \rVert_{L^1(\R^d;X_1)} \lVert g \rVert_{L^p(\R^d;X_2)} .
\end{equation*}
In the following, we will suppress the symbol $\bullet$ if it is clear from the context which multiplication is being employed.
\par
Combining \cite[Theorem~4.1]{Amann-97} and \cite[Corollary~4.4]{Amann-97} we obtain
\begin{lemma}\label{lem:mult}
	Let $\varepsilon > 0$.  Then there exists $C > 0$ such that for all $\mu > 0$ and all $m \in W^{d+1, \infty}(\R^d ; X_1)$ satisfying
	\begin{equation*}
		\lVert m \rVert_{W^{d+1,\infty}} + \max\limits_{\lvert \alpha \rvert \leq d+1} \sup\limits_{\xi \in \R^d} \lvert \xi \rvert^{\lvert \alpha \rvert + \varepsilon}\lVert \partial^\alpha m(\xi) \rVert_{X_1} \leq \mu < \infty
	\end{equation*}
	we have
	\begin{equation*}
		\lVert \mathcal{F}^{-1}m \rVert_{L^1(\R^d;X_1)} \leq C \mu . 
	\end{equation*}
	In particular, it follows from Young's inequality that $m(D) \in \L(L^p(\R^d;X_2), L^p(\R^d;X_0))$ with
	\begin{equation*}
		\lVert m(D) \rVert \leq C\mu .
	\end{equation*}
\end{lemma}
Following \cite[Theorem~2.3]{Amann-97}, we can define a hypocontinuous bilinear mapping $[\cdot,\cdot]_\bullet:\mathcal{S}^\prime(\R^d ;X_1) \times \mathcal{S}(\R^d ;X_2) \to X_0$ by setting
 \[
 	\left[f \otimes x_1,\varphi \otimes x_2 \right]_\bullet = \dual{f,\varphi}_{\Schw^\prime(\R^d) \times \Schw(\R^d)}x_1 \bullet x_2
 \]
for elementary tensors given by $f \in \Schw^\prime(\R^d), \varphi \in \Schw(\R^d), x_1 \in X_1, x_2 \in X_2$ and extending by density. As before, we will suppress the notation of $\bullet$ when it is clear from the context which multiplication is being employed.
It follows that
\begin{equation}\label{eq:parseval}
 	\left[\mathcal{F}f,\varphi\right] = \left[f,\mathcal{F}\varphi\right] ,
\end{equation}
i.e. the Fourier transform is symmetric with respect to this form. If $f \in L^1_\loc(\R^d;X_1)$ and $\varphi \in \D(\R^d;X_2)$, we have
\begin{equation*}
	\left[f,\varphi \right] = \int\limits_{\R^d} f(x)\varphi(x) \diff x .
\end{equation*}
\par
In the following, we specialize to the case $X_2 = X$, $X_1 = X^\prime$ and $x_1 \bullet x_2 = \dual{x_1,x_2}$.  Suppose that $m \in \mathcal{O}_M(\R^d;\L(X))$ and consider $m(D):\mathcal{S}(\R^d ;X) \to \mathcal{S}(\R^d ;X)$. It is clear that the symbol $m^\prime(-\cdot)$ given by $\R^d \ni \xi \mapsto m(-\xi)^\prime \in \L(X^\prime)$ belongs to $\mathcal{O}_M(\R^d;\L(X^\prime))$. Here, $m(-\xi)^\prime$ denotes the adjoint operator of $m(-\xi)$. For any Banach space $Y$ and $f \in \mathcal{S}^\prime(\R^d ;Y)$, $\varphi \in \mathcal{S}(\R^d;Y)$, we set $R\varphi = \varphi(-\cdot)$ and define $Rf \in \mathcal{S}^\prime(\R^d ;Y)$ by $Rf(\psi) = f(R\psi)$ where $\psi \in \mathcal{S}(\R^d ;Y)$. Using that $\mathcal{F}^{-1}f = (2\pi)^{-d}\mathcal{F}Rf$ for $f \in \mathcal{S}^\prime(\R^d;X)$, we deduce from $\eqref{eq:parseval}$ that
\begin{equation*}
	[m'(-D)f,\varphi] = [f,m(D)\varphi] .
\end{equation*}
In particular, if $f \in L^p(\R^d;X^\prime)$ where $1 \leq p < \infty$ and $\varphi \in \D(\R^d)$ we deduce
\begin{equation*}
\int\limits_{\R^d} \dual{(m^\prime(-D)f(x),\varphi(x)} \diff x = \int\limits_{\R^d} \dual{f(x),m(D)\varphi(x)} \diff x .
\end{equation*}
We may therefore deduce the following result, which will be important when relating our observability estimate to null controllability:
\begin{proposition}\label{adjoint-multiplier}
	Let $q$ be such that $p^{-1} + q^{-1} = 1$. Let $X$ be a Banach space such that $X^\prime$ has the Radon-Nikodym property and $1 \leq p  < \infty$. Let $m \in \mathcal{O}_M(\L(X))$ such that
	\begin{equation*}
		\lVert \F^{-1}m \rVert_{L^1(\R^d;\L(X))} < \infty  .
	\end{equation*}
	Then $m(D)^\prime = m^\prime(-D) \in \mathcal{L}(L^{q}(\R^d;X^\prime)$ with
	\begin{equation*}
		\lVert m(D)^\prime \rVert \leq \lVert \F^{-1}m \rVert_{L^1(\R^d;\L(X))} .
	\end{equation*}.
\end{proposition}
Before we proceed with the proof, let us recall that, as in the scalar case, we have the convolution identity
\begin{equation*}
	\F(f*g) = \F f \F g, \quad (f \in L^1(\R^d;\L(X)), g \in L^p(\R^d;X))\,.
\end{equation*}
This identity can be verified first on elementary tensors and then established in the general case by a density argument. Thus, it follows from Fourier inversion, the above identity and Young's inequality that
\begin{equation*}
	\lVert m(D) \rVert_{\L(L^p(\R^d;X))} \leq \lVert \F^{-1}m \rVert_{L^1(\R^d;\L(X))}\,.
\end{equation*}
\begin{proof}
	From  the Radon-Nikodym property of $X'$ we have that $L^p(\R^d;X)^\prime \simeq L^{q}(\R^d;X^\prime)$ and
	\begin{equation*}
		\dual{f,g}_{ L^{q}(\R^d;X^\prime) \times L^p(\R^d;X)} = \int\limits_{\R^d} \dual{f(x),g(x)}_{X^\prime \times X} \diff x, \quad (f,g) \in L^{q}(\R^d;X^\prime) \times L^p(\R^d;X) .
	\end{equation*}
	In particular, if $\varphi \in \mathcal{D}(\R^d ;X)$ it holds that
	\begin{equation*}
		\dual{m^\prime(-D)f,\varphi}_{ L^{q}(\R^d;X^\prime) \times L^p(\R^d;X)} = \dual{f,m(D)\varphi}_{ L^{q}(\R^d;X^\prime) L^p(\R^d;X)} .
	\end{equation*}
	Now, let $(f,g) \in L^{q}(\R^d;X^\prime) \times L^p(\R^d;X)$. Since $\Schw(\R^d;X)$ is dense in $L^p(\R^d;X)$, we can choose a sequence $(\varphi_k)_{k=0}^\infty \in \Schw(\R^d;X)^{\N}$ such that
	\begin{equation*}
		\lVert \varphi_k - g \rVert_{L^p(\R^d;X)} \to 0 \quad (k \to \infty) .
	\end{equation*}
	Thus, $m(D)\varphi_k \to m(D)g$ in $L^p(\R^d;X)$ as $k \to \infty$ and since
	\begin{equation*}
	\dual{m^\prime(-D)f,g}_{ L^{q}\times L^p} = \dual{f, m(D)\varphi_k}_{ L^{q}\times L^p} + \dual{m^\prime(-D)f,g - \varphi_k}_{ L^{q} \times L^p}
	\end{equation*}
	it follows that
	\begin{equation*}
	\dual{m^\prime(-D)f,g}_{ L^{q}\times L^p} = \dual{f, m(D)g}_{ L^{q}\times L^p}
	\end{equation*}
	which proves $m(D)^\prime = m^\prime(-D)$.
	\par
	Let $x \in \R^d$. It follows that for $(\ell,v) \in X^\prime \times X$
	\begin{align*}
		\dual{[(\F^{-1} m)(x)]^\prime \ell, v} = \dual {\ell,  (\F^{-1} m)(x)v}  &= \frac{1}{(2\pi)^d} \int\limits_{\R^d} e^{ \ii \xi \cdot x}\dual{\ell,m(\xi)v} \diff \xi \\ &= \frac{1}{(2\pi)^d} \int\limits_{\R^d} e^{\ii \xi \cdot x}\dual{m(\xi)^\prime\ell,v} \diff \xi  \\
		& = \dual{(\F^{-1} m^\prime)(x)\ell, v} .
	\end{align*}
	Therefore, since $m(D)^\prime = m^\prime(-D)$  we obtain
	\begin{equation*}
		\lVert \mathcal{F}^{-1}m'(-\cdot) \rVert_{L^1(\R^d;\L(X^\prime))} = \lVert(\mathcal{F}^{-1}m)^{\prime}\rVert_{L^1(\R^d;\L(X^\prime))} = \lVert\mathcal{F}^{-1}m\rVert_{L^1(\R^d;\L(X))} \leq k
	\end{equation*}
	and the result follows.
\end{proof}
\section{Logvinenko-Sereda theorem for vector-valued functions}
Let $X$ be a Banach space with norm $\lVert \cdot \rVert_X$. In order to formulate our main result we recall the notion of a $(\rho,L)$-thick subset $\thickset$ of $\R^d$, and the notation $\Pi_\lambda$ for the parallelepiped with side lengths $\lambda_i$, $i \in \{1,2,\ldots , d\}$, cf.\ Eqs.~\eqref{eq:thick-set} and \eqref{eq:parallelepiped} in the introduction. For $f \in L^p (\R^d ; X)$ we denote by $\F f$ its Fourier transform, cf.\ Section~\ref{sec:preliminaries}.
\begin{theorem}\label{thm:LS}
There exists a constant $C_{\mathrm{LS}} \geq 1$ such that for all $p \in [1,\infty]$, all $\lambda \in (0,\infty)^d$, all $f \in L^p (\R^d ; X)$ with $\spt \mathcal{F} f \subseteq \Pi_\lambda$, all $\rho > 0$, all $L \in (0,\infty)^d$, and all $(\rho , L)$-thick sets $\thickset \subseteq \R^d$ we have 
\begin{equation*} 
\lVert \Eins_E f \rVert_{L^p (\R^d ; X)} \geq 
\left( \frac{\rho}{C_{\mathrm{LS}} } \right)^{C_{\mathrm{LS}} (d + L\cdot \lambda) }  \lVert f \rVert_{L^p (\R^d ; X)} .
\end{equation*}
\end{theorem}
In the case where $X = \C$ this theorem was originally proven by Logvinenko and Sereda in \cite{LogvinenkoS-74} and significantly improved by Kovrijkine in \cite{Kovrijkine-00,Kovrijkine-01}. For further references concerning the case $X=\C$ we refer to the introduction. Let us stress that the essential improvement of Theorem~\ref{thm:LS} is reflected in the (possible) infinite dimensionality of the Banach space $X$. To this end, let us consider the following example.
\begin{example}
 Let $I$ be a countable index set, and consider for $i \in I$ the functions $f_i \in L^p (\R^d)$ with $\spt \F f_i \subset \Pi_\lambda$ for some $\lambda \in (0,\infty)$. Thus, the classical Logvinenko-Sereda theorem (i.e.\ $X=\C$) applies to each $f_i$ separately. Now we assume that the pointwise supremum $g : \R^d \to \R$, 
 \[
  g (x) = \sup \{\lvert f_i (x) \rvert \colon i \in I\} 
 \]
 is in $L^p (\R^d)$.
 Then, Theorem~\ref{thm:LS-C} with $X = \ell^\infty (I)$ applied to the function $f : \R^d \to \ell^\infty (I)$, $(f(x))_i = f_i (x)$, gives
 \[
 \lVert \Eins_E g \rVert_{L^p (\R^d)}
 = \left(\int_E  \lVert f(x) \rVert_{\ell^\infty (I)}^p \diff x\right)^{1/p}
 \geq 
\left( \frac{\rho}{C_{\mathrm{LS}} } \right)^{C_{\mathrm{LS}} (d + L\cdot \lambda) }\lVert  g \rVert_{L^p (\R^d)} .
 \]
Indeed, if the index set $I$ is finite, it is feasible to conclude this estimate directly from the classical Logvinenko-Sereda theorem ($X=\C$) with a constant depending on the cardinality of $I$. If the cardinality of $I$ is infinite, our Theorem~\ref{thm:LS} applies. 
\end{example}
For the proof of Theorem~\ref{thm:LS} we will follow the main strategy given in \cite{Kovrijkine-00}. However, in order to deal with Banach space valued functions instead of $\C$-valued functions we shall need two preparatory results, i.e.\ Proposition~\ref{prop:funktionentheorie} and Proposition~\ref{prop:Bernstein}, which we formulate next. The final proof of Theorem~\ref{thm:LS} is postponed to the appendix.
\par
For $z \in \C$ and $r > 0$ we denote by $\disc{z}{r} \subseteq \C$ the open disc of radius $r$ centered at $z$. As well, let $\ball{x}{r} \subseteq \R^d$ be the ball of radius $r$ centered at $x \in \R^d$. If $z = 0$ or $x = 0$, respectively, we simply write $\discc{r}$ or $\ballc{r}$.
\begin{proposition} \label{prop:funktionentheorie}
There exists a constant $C_1 \geq 1$ such that for all closed intervals $I \subseteq \R$ with $0 \in I$ and $\lvert I \rvert = 1$, all analytic functions $f:\discc{6} \to X$ satisfying
\[
\sup_{z \in \discc{5}}\lVert f(z) \rVert_X \leq M 
\quad\text{and}\quad
\sup_{x \in I} \lVert f(x) \rVert_X \geq 1 
\]
for some $M>0$, and all measurable $A \subseteq I$ we have
\begin{equation}\label{eq:com_vec}
\sup_{x \in A} \lVert f(x) \rVert_X \geq	\left( \frac{\lvert A \rvert}{C_1} \right)^{\frac{\ln(M)}{\ln(2)}} \sup_{x \in I} \lVert f(x) \rVert_X.
\end{equation}
\end{proposition}
\begin{proposition} \label{prop:Bernstein}
There exists a constant $C_2 > 0$ such that for all $\lambda \in (0,\infty)^d$, all $p \in [1,\infty]$, all $f \in L^p(\R^d;X)$ with $\spt \mathcal{F} f \subseteq \Pi_\lambda$ and all $\alpha \in \N_0^d$ we have 
\begin{equation*}
	\lVert \partial^\alpha f \rVert_{L^p(\R^d;X)} \leq C_2^{\lvert \alpha \rvert}\lambda^\alpha  \lVert f \rVert_{L^p(\R^d;X)} .
\end{equation*}
\end{proposition}
\begin{proof}[Proof of Proposition~\ref{prop:funktionentheorie}]
Without loss of generality we assume that $\lvert A \rvert > 0$. Since $I$ is closed and $\lVert f(\cdot) \rVert_X$ is continuous on $I$, there exists $x_0 \in I$ such that $\sup_{x \in I} \lVert f(x) \rVert_X = \lVert f(x_0) \rVert_X$. By a consequence of the Hahn-Banach theorem, we can find $x^\prime \in X^\prime$ such that $\lVert x^\prime \rVert_{X^\prime} = 1$ and
\begin{equation*}
	\dual{x^\prime,f(x_0)} = \lVert f(x_0) \rVert_X .
\end{equation*}
The function $\varphi:\discc{5} \to \C$ given by $\varphi = \dual{x^\prime,f(\cdot +x_0)}$ is analytic and we have
\begin{equation*}
	\lvert \varphi(0) \rvert = \lVert f(x_0) \rVert_X \geq 1
\end{equation*}
as well as
\begin{equation*}
	\lvert \varphi(z) \rvert \leq \lVert f(z + x_0) \rVert_X \leq M 
\end{equation*}
for all $z \in \discc{4}$.
Moreover, the sets $I - x_0$ and $A -x_0$ are such that $A -x_0 \subseteq I -x_0$, $A-x_0$ is of positive measure by assumption and $0 \in I - x_0$. Applying Lemma~1 in \cite{Kovrijkine-01} with $\varphi$ as above as well as $I$ and $A$ replaced by $I -x_0$ and $A - x_0$ respectively, we obtain that there exists a constant $C_1 > 0$ such that
\begin{equation*} 
	 \sup_{x \in A -x_0} \lvert \varphi(x) \rvert \geq \left( \frac{\lvert A \rvert}{C_1} \right)^{\frac{\ln(M)}{\ln(2)}}  \sup_{x \in I -x_0} \lvert \varphi(x) \rvert .
\end{equation*}
Inequality~\eqref{eq:com_vec} now follows from
\begin{align*}
   \sup_{x \in A} \lVert f(x) \rVert_X 
   \geq 
   \sup_{x \in A} \lvert \dual{x^\prime,f(x)} \rvert
   &=
   \sup_{x \in A - x_0} \lvert \varphi(x) \rvert \\
   &\geq
     \left( \frac{\lvert A \rvert }{C_1} \right)^{\frac{\ln(M)}{\ln(2)}} \lvert \varphi(0) \rvert
     =
      \left( \frac{\lvert A \rvert }{C_1} \right)^{\frac{\ln(M)}{\ln(2)}} 	\sup_{x \in I} \lVert f(x) \rVert_X . \qedhere
\end{align*}
\end{proof}
\begin{proof}[Proof of Proposition~\ref{prop:Bernstein}]
The proof is an adaption of the classical proof, as it can be found for example in \cite{Wolff-03}, to the vector-valued setting. We only prove the assertion in the case $\lvert \alpha \rvert = 1$. The case $\lvert \alpha \rvert = 0$ is trivial, and the case $\lvert \alpha \rvert > 1$ follows by induction. We choose a real-valued function $\varphi\in \Schw(\R^d)$ such that $0 \leq \varphi \leq 1$ as well as $\varphi = 1$ on $[-1/2 , 1/2]^d$ and define $\varphi_\lambda = \varphi(T_\lambda\cdot)$ where
\begin{equation*}
	T_\lambda:\R^d \to \R^d, \quad (x_1,\dots,x_d) \mapsto (x_1/\lambda_1,\dots,x_d/\lambda_d).
\end{equation*}
Clearly, $\varphi_\lambda = 1$ on $\Pi_\lambda$, $\F^{-1} \varphi_\lambda = \lambda_1\lambda_2 \ldots \lambda_d (\F^{-1}\varphi) (T_\lambda^{-1}\cdot)$. Moreover, since the usual convolution identity also holds in the vector-valued setting, we have $f = \F^{-1}(\varphi_\lambda \F{f}) = (\F^{-1} \varphi_\lambda) * f$. From Young's inequality we conclude for all $j \in \{1,2,\ldots , d\}$ that
\begin{equation*}
	\lVert\partial_j f\rVert_{L^p(\R^d;X)} 
	= 
	\lVert(\partial_j \F^{-1} \varphi_\lambda) * f\rVert_{L^p(\R^d;X)} 
	\leq 
	\lVert\partial_j \F^{-1} \varphi_\lambda\rVert_{L^1(\R^d)}\lVert f \rVert_{L^p(\R^d;X)} .
\end{equation*}
Since
\[
 \lVert\partial_j \F^{-1} \varphi_\lambda\rVert_{L^1(\R^d)}
 = \lambda_j\lVert\lambda_1\lambda_2 \ldots \lambda_d(\partial_j\F^{-1} \varphi)(T_\lambda^{-1}\cdot)\rVert_{L^1(\R^d)} = \lambda_j \lVert\partial_j \F^{-1} \varphi\rVert_{L^1(\R^d)},
\]
the assertion (in the case $\lvert \alpha \rvert = 1$) follows with $C_2 = \sup_{j = 1,\dots,d} \lVert\partial_j\F^{-1} \varphi\rVert_{L^1(\R^d)}$.
\end{proof}
\section{Control theory for normally elliptic operators on Banach spaces}
\subsection{Normally elliptic operators and their semigroups}
In \cite{Amann-01}, the notion of normal ellipticity has been introduced for operators with variable, $\L(X)$-valued, non-smooth coefficients and it was shown that their negatives generate analytic semigroups on $L^p(\R^d;X)$. This general framework is technically challenging and involves, for example, Besov spaces of vector-valued functions. In what follows, we consider normally elliptic operator $A$ with \emph{constant} coefficients only. As a consequence, certain proofs of \cite{Amann-01} simplify and we obtain stronger results. In particular, using ideas from \cite{Amann-97, Amann-01, Amann-19} we show that:
\begin{enumerate}[(i)]
	\item $-A_p$, the part of $A$ in $L^p(\R^d;X)$, is a semigroup generator and one can represent the resulting semigroup as a Fourier multiplier. This is suggested by \cite[Remark 7.5]{Amann-97}. Here, we give a full proof of this result.
	\item The derivatives of the symbol of this multiplier decay exponentially. This is the content of Lemma~\ref{lem:symbol_est} which is the crucial result of this section for our application to control theory. In Proposition~3.5.7 of \cite{Amann-19} a similar estimate is given, but with polynomial decay.
\end{enumerate}
Let $X$ be a Banach space and $d,m \in \N$. For given coefficients $a_\alpha \in \L(X)$ where $\alpha$ ranges over all multi-indices with $\lvert \alpha \rvert \leq m$, consider the polynomial $a:\R^d \to \L(X)$,
\begin{equation*}
	a(\xi) = \sum_{\lvert \alpha \rvert \leq m} a_\alpha \xi^\alpha .
\end{equation*}
We suppose that $a$ has degree $m$, meaning there exists a multi-index $\alpha \in \N_0^d$ such that $\abs{\alpha} = m$ and $a_\alpha \neq 0$. The set of all polynomials of this type is denoted by $\mathcal{P}_m(\R^d;\L(X))$. The associated Fourier multiplier $A = a(D)$ is a differential operator acting on $\Schw'(\R^d;X)$, see Section~\ref{sec:preliminaries}. The principal symbol of $A$ is the polynomial $a_m:\R^d \to \L(X)$, 
\begin{equation*}
	a_m(\xi) = \sum_{\lvert \alpha \rvert = m} a_\alpha \xi^\alpha .
\end{equation*}
Let $\kappa \geq 1$, $\vartheta \in [0,\pi)$ and $\omega  \in \R$. We write
\begin{equation*}
	\Sigma_{\vartheta,\omega}= \{ z \in \C: \lvert \arg(z-\omega) \rvert \leq \vartheta \}\cup\{0\} .
\end{equation*}
Given a linear operator $T \in \L(X)$, we denote its resolvent set by $\rho(T)$. We say that a differential operator $A$ is $(\kappa,\vartheta,\omega)$-elliptic if for all $\xi \in \R^d$ with $\lvert \xi \rvert = 1$ it holds that
\begin{equation*}
	\rho(-a_m(\xi)) \supseteq \Sigma_{\vartheta,\omega} 
\end{equation*}
and for all $\lambda \in \Sigma_{\vartheta,\omega}$,
\begin{equation*}
	\lVert(\lambda + a_m(\xi))^{-1}\rVert \leq \frac{\kappa}{1 + \lvert \lambda-\omega \rvert} .
\end{equation*}
We say that $A$ is normally elliptic (with symbol $a$) if it is $(\kappa,\pi/2,0)$-elliptic and call $\kappa$ a ellipticity constant of $A$.
\par
Let $1 \leq p < \infty$. We denote by $A_p$ the part of $A$ in $L^p(\R^d;X)$, that is 
\begin{equation*}
	\dom(A_p) = \{f \in L^p(\R^d;X): Af \in L^p(\R^d;X)\}, \quad A_pf = Af .
\end{equation*}
\begin{remark}
Suppose that $A$ is $(\kappa,\vartheta,\omega)$-elliptic. By homogeneity we obtain for all $\xi \neq 0$ and $\lambda \in \rho(-a_m(\xi))$
\begin{equation*}
(\lambda + a_m(\xi))^{-1} = (\lambda + \lvert \xi \rvert^m a_m(\xi/\lvert \xi \rvert))^{-1} = \lvert \xi \rvert^{-m}(\lvert \xi \rvert^{-m}\lambda + a_m(\xi/\lvert \xi \rvert))^{-1}.
\end{equation*}
Therefore, if $\xi \neq 0$, and $\lambda \in \lvert \xi \rvert^m\Sigma_{\vartheta, \omega} = \Sigma_{\vartheta,\omega\lvert \xi \rvert^m}$, then
\begin{equation*}
\lVert (\lambda + a_m(\xi))^{-1}\rVert \leq \frac{\kappa}{\lvert \xi \rvert^m+ \lvert \lambda - \omega \lvert \xi \rvert^m \rvert } .
\end{equation*}
\end{remark}
\begin{proposition}\label{prop:elliptic}
	If $A$ is normally elliptic with ellipticity constant $\kappa$, there exist $\varphi > \pi/2 $ and $M > 0$ as well as $\mu < 0$ such that $A$ is $(M,\varphi,\mu)$-elliptic. Moreover, we can choose $(M,\varphi,\mu) = (2\kappa + 1,\pi - \arctan(2\kappa),-1 / (2\kappa))$.
\end{proposition}
\begin{proof}
 	Suppose that $T \in \L(X)$ and $K > 0$, $z \in \rho(-T)$ are such that
	\begin{equation*}
		\lVert (z+T)^{-1} \rVert \leq K .
	\end{equation*}
	Then it follows from the usual Neumann series argument that $D(z,K^{-1}) \subseteq \rho(-T)$ and we have for all $w \in D(z,K^{-1})$ that
	\begin{equation*}
		(w+T)^{-1} = \sum_{n=0}^\infty (z-w)^{n}(z + T)^{-1 -n} ,
	\end{equation*}
	which leads to the estimate
	\begin{equation*}
		\lVert (w + T)^{-1} \rVert \leq K\sum_{n=0}^{\infty} \lvert z-w \rvert^{n}K^{n} = \frac{K}{1 - \lvert z-w \rvert K} .
	\end{equation*}
	In particular, if $w \in \overline{D}(z,(2K)^{-1})$ we get
	\begin{equation*}
		\lVert (w+T)^{-1} \rVert \leq 2K .
	\end{equation*}
	Now, let $A$ be normally elliptic. Fix $\sigma \in \R$ and $\tau \in [0,(1 + \lvert \sigma \rvert)/(2\kappa)]$. Clearly, we have $-\tau + \ii \sigma$ in $\overline D(\ii \sigma, (1 + \lvert \sigma \rvert)/(2\kappa))$. Let $\lvert \xi \rvert = 1$. Applying the above considerations to $T = a_m(\xi)$, we obtain
	\begin{equation*}
		\lVert (-\tau + \ii\sigma + a_m(\xi))^{-1} \rVert \leq \frac{2\kappa}{1 + \lvert \sigma \rvert} .
	\end{equation*}
	Furthermore, since
	\begin{equation*}
		1 + \lvert -\tau + \ii \sigma + \frac{1}{2\kappa} \rvert \leq 1 + \lvert -\tau + \frac{1}{2\kappa} \rvert + \lvert \sigma \rvert 
		\leq 1 + \frac{1}{2\kappa}(1 + \lvert \sigma \rvert) + \lvert \sigma \rvert 
		\leq \frac{2\kappa + 1}{2\kappa}(1 + \lvert \sigma \rvert)
	\end{equation*}
	we obtain
	\begin{equation*}
		\frac{2\kappa}{1 + \lvert \sigma \rvert} \leq \frac{2\kappa + 1}{1 + \lvert -\tau + \ii \sigma + \frac{1}{2\kappa} \rvert } .
	\end{equation*}
	Moreover, we have
	\begin{equation*}
		\left\lvert \arg\left(-\tau + \ii\sigma + \frac{1}{2\kappa}\right)\right\rvert \leq \left\lvert \arg\left( -\frac{\lvert \sigma \rvert}{2\kappa} + \ii\sigma \right)\right\rvert \leq \pi - \arctan(2\kappa),
	\end{equation*}
	where the argument of a complex number has to be understood as an element of $[-\pi,\pi)$.
	Since
	\begin{equation*}
		\Sigma_{\pi - \arctan(2\kappa),-\frac{1}{2\kappa}} \cap \{\lambda \in \C: \real(\lambda) \leq 0 \} = \left\{ -\tau + \ii \sigma:~ \sigma \in \R,~ \tau \in \left[0,\frac{1 + \lvert \sigma \rvert}{2\kappa}\right]  \right\} ,
	\end{equation*}
	we conclude that for all $\lambda \in \Sigma_{\pi - \arctan(2\kappa),-1/(2\kappa)} \cap \{\lambda \in \C: \real(\lambda) \leq 0 \}$ we have
	\begin{equation*}
			\lVert (\lambda + a_m(\xi))^{-1} \rVert \leq \frac{2\kappa +1}{1 + \lvert \lambda+ 1 / (2\kappa)\rvert} .
	\end{equation*}
It is easy to see that this estimate also holds if $\real(\lambda) > 0$. The latter inequality implies that $A$ is $(2\kappa + 1,\pi - \arctan(2\kappa),-1 / (2\kappa))$-elliptic.	
\end{proof}
For all $n \geq 0$, all $p \in \mathcal{P}_n(\R^d;\L(X))$, and all multi-indices $\alpha \in \N_0^d$ we define 
	\begin{equation*}
		N_\alpha(p) = \max_{\beta \leq \alpha} \sup_{\xi \in \R^d} \frac{\lVert \partial^\beta p(\xi) \rVert}{(1 + \lvert \xi \rvert)^{n - \lvert \beta \rvert}} ,
	\end{equation*}
	where for multi-indices $\alpha,\beta \in \N_0^d$ we write $\beta \leq \alpha$ if $\beta_i \leq \alpha_i$ for all $i \in \{1,2,\ldots , d\}$.
\begin{proposition}\label{prop:pertubation}
Suppose that $A$ is normally elliptic with ellipticity constant $\kappa$. Then there exist $\varphi, \gamma, \omega, M > 0$ such that for all $\xi \in \R^d$, and all $\lambda \in \Sigma_{\varphi,-\gamma\lvert \xi \rvert^m + \omega}$ we have
\begin{equation*}
	\lVert (\lambda + a(\xi))^{-1} \rVert \leq \frac{M}{\lvert \xi \rvert^m + \lvert\lambda + \gamma\lvert \xi\rvert^m \rvert} .
\end{equation*}
The parameters $\varphi, \gamma$ depend only on $a_m$ while $\omega$ depends on $a_m$ and $N_0(a-a_m)$. Moreover, we can choose
\begin{equation*}
	M = 4\kappa +2 .
\end{equation*}
\end{proposition}
\begin{proof}
We employ the following well-known perturbation result based on the Neumann series: If $T,S \in \L(X)$ such that
\begin{equation*}
	\lVert ST^{-1} \rVert \leq \frac{1}{2}
\end{equation*}
then $T+S$ is invertible and
\begin{equation*}
\lVert (T + S)^{-1} \rVert \leq 2\lVert T \rVert .
\end{equation*}
We infer from Proposition~\ref{prop:elliptic} that there exist constants $C,\varphi, \gamma > 0$ depending only on $a_m$ such that for all $\lambda \in \Sigma_{\varphi,-\gamma\lvert \xi \rvert^m}$
\begin{equation*}
\lVert (\lambda + a_m(\xi))^{-1} \rVert \leq \frac{C}{\lvert \xi \rvert^m + \lvert \lambda +\gamma\lvert \xi \rvert^m\rvert} .
\end{equation*}
We note that $a - a_m$ has degree $m-1$. For a sufficiently large $\omega > 0$, we obtain for all $\lambda \in \Sigma_{\varphi,-\gamma\lvert \xi \rvert^m + \omega}$
\begin{equation*}
\lVert (a(\xi) - a_m(\xi))(\lambda + a_m(\xi))^{-1} \rVert \leq \frac{CN_0(a-a_m)(1 + \lvert \xi \rvert)^{m-1}}{\lvert \xi \rvert^m + \lvert \lambda +\gamma\lvert \xi \rvert^m \rvert} \leq \frac{1}{2} .
\end{equation*}
From the perturbation result and Proposition~\ref{prop:elliptic}, we obtain the claimed inequality.
\end{proof}
Let $A$ be a normally elliptic operator. The above Proposition implies that for all $\xi \in \R^d$ and all $\lambda \in \Sigma_{\varphi,-\gamma\lvert \xi \rvert^m + \omega}$ we have
\begin{equation} \label{eq:sectorial}
\lVert (\lambda + a(\xi))^{-1} \rVert \leq \frac{1}{\sin(\varphi)}\frac{M}{\lvert \lambda + \gamma \lvert \xi \rvert^m - \omega \rvert} .
\end{equation}
This can be seen as follows: Using the notation $\lambda + \gamma \lvert \xi \rvert^m - \omega= r \euler^{\ii \psi}$, where $r > 0$ and $\psi \in [- \varphi , \varphi]$ we find
\begin{align*}
\frac{\lvert \lambda + \gamma \lvert \xi \rvert^m - \omega \rvert}{\lvert \xi \rvert^m + \lvert \lambda + \gamma\lvert \xi \rvert^m}\rvert &\leq \frac{\abs{\lambda + \gamma \abs{\xi}^m - \omega}}{\abs{\lambda + \gamma\abs{\xi}^m - \omega + \abs{\xi}^m + \omega}} \\
&\leq \sup_{r > 0} \sup_{\psi \in [-\varphi , \varphi]} \frac{\abs{r \euler^{\ii \psi}}}{\abs{r \euler^{\ii \psi} + \abs{\xi}^m + \omega}} \\
&\leq \frac{r}{\imag(\abs{r \euler^{\ii \psi} + \abs{\xi}^m + \omega})} \leq \frac{1}{\sin(\varphi)} .
\end{align*}
This implies Ineq.~\eqref{eq:sectorial}. Thus, $-a(\xi)$ is a sectorial operator in the sense of \cite[Definition~2.0.1]{Lunardi-95}. Hence, $-a(\xi)$ generates for all $\xi \in \R^d$ an analytic semigroup on $X$ which we denote by $(S_t(\xi))_{t \geq 0}$. Consequently, there exists a $C > 0$ such that for all $\xi \in \R^d$ and all $t \geq 0$ we have
\begin{equation}\label{eq:exp}
\lVert S_t(\xi) \rVert \leq C\euler^{\omega t - \gamma \abs{\xi}^m t} .
\end{equation}
Note that the constant $C$ is independent of $\xi$ since $M$ and $\varphi$ in Ineq.~\eqref{eq:sectorial} are independent of $\xi$.
\begin{lemma}\label{lem:symbol_est}
	Let $A$ be a normally elliptic operator with symbol $a$ and denote for each $\xi \in \R^d$ the semigroup generated by $-a(\xi)$ by $(S_t(\xi))_{t \geq 0}$. Then there exist $\mu, \omega > 0$ depending only on $a_m$ such that for each multi-index $\alpha$ there exists a constant $K_\alpha > 0$ such that for all $\xi \in \R^d$ and $t \geq 0$ it holds that
	\begin{equation}\label{eq: s_a}
		\lVert \partial^\alpha S_t(\xi) \rVert \leq K_\alpha \euler^{\omega t - \mu\abs{\xi}^m t} .
	\end{equation}
	The constant $K_\alpha$ can be chosen to depend only on the principal symbol $a_m$ and $N_\alpha(a)$.
\end{lemma}
\begin{proof}
	Let $\xi \in \R^d$. Since $A$ is normally elliptic, Proposition~\ref{prop:elliptic} implies that there exist $\tilde M,\lambda,\gamma > 0$ and $\varphi \in (\pi / 2,\pi)$ such that
	\begin{equation}\label{eq:ell_constants}
		 \lVert (\lambda - \gamma\abs{\xi}^m + \omega + a(\xi))^{-1}\rVert \leq \frac{\tilde M}{\abs{\xi}^m + \abs{\lambda + \omega }}, \quad (\lambda \in \Sigma_{\varphi,0}) .
	\end{equation}
	We set $b(\xi) = -a(\xi) + \gamma\abs{\xi}^m - \omega$.
	Due to $\abs{\lambda + \omega} \geq \sin(\varphi)\abs{\lambda}$ for $\lambda \in \Sigma_{\varphi,0}$ and setting $M = \tilde M (\sin(\varphi))^{-1}$ it follows that
	\begin{equation*}
		\lVert (\lambda - b(\xi))^{-1} \rVert = \lVert (\lambda - \gamma\abs{\xi}^m + \omega + a(\xi))^{-1} \rVert \leq \frac{M}{\abs{\xi}^m + \abs{\lambda}}, \quad (\lambda \in \Sigma_{\varphi,0}) .
	\end{equation*}
	Write $(T_t(\xi))_{t \geq 0}$ for the semigroup generated by $b(\xi)$. It is clear that
	\begin{equation}\label{eq:ba}
		T_t(\xi) = \euler^{-\omega t + \gamma t\abs{\xi}^m }S_t(\xi) .
	\end{equation}
	Let $\alpha$ be a multi-index. We show that there exists a constant $\tilde M_\alpha > 0$ such that
	\begin{equation}\label{eq:s_b1}
		\lVert \partial^\alpha T_t(\xi) \rVert \leq \tilde M_\alpha(t(1 +\abs{\xi})^{m-1} + t^\abs{\alpha}(1+\abs{\xi})^{(m-1)\abs{\alpha}}) \quad (\xi \in \R^d, t \geq 0)
	\end{equation}
	holds.
	For $\abs{\alpha} = 0$ this is straightforward by Ineq.~\eqref{eq:exp}. Therefore, we assume that $\abs{\alpha} \geq 1$ in the following.
	\par
		Let $r > 0$. Consider the contour
		\begin{equation*}
		\Gamma= \euler^{i\varphi}[r,\infty) \cup (r\T \cap \Sigma_{\varphi,0}) \cup \euler^{-i\varphi}[r,\infty) .
		\end{equation*}
		with positive orientation, where $\T$ denotes the unit circle in $\C$.
		Let $\alpha$ be a multi-index. For every $t \geq 0$, we consider the functions
		\begin{equation*}
		T^{(\alpha)}_t:\R^d \to \L(X), \quad T^{(\alpha)}_t(\xi) = \frac{1}{2\pi\ii} \int\limits_\Gamma \euler^{t\lambda}\partial^\alpha(\lambda - b(\xi))^{-1} \diff \lambda .
		\end{equation*}
		For the sake of simplicity we will write $b$ instead of $b(\xi)$. Since
		\begin{equation*}
			\partial_j(\lambda - b)^{-1} = (\lambda - b)^{-1}(\partial_j b)(\lambda - b)^{-1} ,
		\end{equation*}
		it follows by induction on the length of $\alpha$ that $\partial^\alpha(\lambda- b)^{-1}$ is a finite sum of terms having the form
		\begin{equation*}
		Q(\beta_1,\beta_2,\dots,\beta_\nu,b,\lambda) = 	(\lambda -b)^{-1}(\partial^{\beta_1}b)(\lambda - b)^{-1}(\partial^{\beta_2}b)\dots(\lambda - b)^{-1}(\partial^{\beta_\nu}b)(\lambda - b)^{-1}
		\end{equation*}
		where $1 \leq  \nu \leq \abs{\alpha}$ and $\beta_1,\beta_2,\dots,\beta_\nu$ are nonzero multi-indices of length $\leq m$ such that $\beta_1 + \beta_2+ \dots +\beta_\nu = \alpha$, see Eq.~(7.4) in \cite{Amann-97}.
 We have the estimate
	\begin{align}
    \lVert Q(\beta_1,\beta_2,\dots,\beta_\nu,b,\lambda)\rVert 
	&\leq 
	\lVert(\lambda - b)^{-1}\rVert^{\nu +1}\prod_{\mu = 1}^{\nu} \lVert\partial^{\beta_\mu}b\rVert \nonumber \\
		&\leq \frac{N_\alpha(b)   M^{\nu+1}}{(\abs{\xi}^m + \abs{\lambda})^{\nu + 1}}\prod_{\mu = 1}^{\nu} (1 + \abs{\xi})^{m - \abs{\beta_\mu}} \nonumber \\
		&\leq \frac{N_\alpha(b) M^{\nu+1}(1+\abs{\xi})^{\nu m-\abs{\alpha}}}{(\abs{\xi}^m + \abs{\lambda})^{\nu + 1}} . \label{eq:QQ}
	\end{align}
	Now, for $\lambda = \rho \euler^{\pm \ii\psi}$ with $\rho > 0$ and $\psi \in [-\varphi , \varphi]$ it follows that
	\begin{equation*}
		\lVert \euler^{t\lambda }Q(\beta_1,\beta_2,\dots,\beta_\nu,b,\lambda) \rVert \leq  \frac{N_\alpha(b)  M^{\nu+1}(1+\abs{\xi})^{\nu m-\abs{\alpha}}}{(\abs{\xi}^m + \rho)^{\nu+1}}\euler^{t\rho\cos(\psi)} .
	\end{equation*}
	Thus, it follows that
	\begin{align*}
			\biggl\lVert \int\limits_\Gamma \euler^{t\lambda}& Q(\alpha_1,\alpha_2,\dots,\alpha_\nu,b,\lambda) \diff \lambda \biggr\rVert \\
			&\leq N_\alpha(b) M^{\nu+1}(1+\abs{\xi})^{\nu m-\abs{\alpha}} \left(2\int\limits_r^\infty \frac{\euler^{t\rho\cos(\varphi)}}{(\abs{\xi}^m +\rho)^{\nu +1}} \diff\rho +  \frac{2\varphi r \euler^{tr}}{(\abs{\xi}^m + r)^{\nu+1}} \right) \\
			&\leq N_\alpha(b) M^{\nu+1}(1+\abs{\xi})^{\nu m- \abs{\alpha}}\left(\frac{2\euler^{tr\cos(\varphi)}}{ t\abs{\cos(\varphi)}(\abs{\xi}^m+r)^{\nu + 1}} +  \frac{2\varphi r \euler^{tr}}{(\abs{\xi}^m + r)^{\nu+1}}\right)  .
	\end{align*}
		Choosing $r = 1/t$ and noting that
		\begin{equation*}
		\frac{1}{t(\abs{\xi}^m + \frac{1}{t})^{\nu + 1}} = \frac{t^\nu}{t^{\nu + 1}(\abs{\xi}^m + \frac{1}{t})^{\nu + 1}} = \frac{t^\nu}{(t\abs{\xi}^m + 1)^{\nu + 1}}
		\end{equation*} we obtain that there exists a constant $C_\varphi > 0$ depending only on $\varphi$ such that
		\begin{align*}
		\biggl\lVert \int\limits_\Gamma \euler^{t\lambda}Q(\alpha_1,\alpha_2,\dots,\alpha_\nu,b,\lambda) \diff \lambda \biggr\rVert &\leq C_\varphi N_{\alpha}(b)  M^{\nu+1}  \frac{t^\nu(1+\abs{\xi})^{\nu m-\abs{\alpha}}}{(t\abs{\xi}^m + 1)^{\nu +1}} \\ &\leq C_\varphi N_\alpha(b) M^{\nu+1} t^\nu (1+\abs{\xi})^{\nu m-\abs{\alpha}} . 
		\end{align*}
		Denote by $C$ a generic constant depending only on $d$ and $m$ whose value may change from line to line. Since $T^{(\alpha)}_t(\xi)$ is a finite sum of terms such as the one above with $1 \leq \nu \leq \abs{\alpha}$ it follows that there exists a constant $C$ such that if we set
		\begin{equation*}
			K_0= C_\varphi M^{\abs{\alpha}+1} N_{\alpha}(b),
		\end{equation*}
		we obtain for all $\xi \in \R^d$ and $t \geq 0$
		\begin{equation}\label{eq:valpha}
				\bigl\lVert T^{(\alpha)}_t(\xi) \bigr\rVert \leq CK_0(t(1 +\abs{\xi})^{m-1} + t^\abs{\alpha}(1+\abs{\xi})^{(m-1)\abs{\alpha}}) .
		\end{equation}
			In particular, in view of the Dunford-Riesz representation
					\begin{equation*}
						T_t(\xi) = \frac{1}{2\pi\ii} \int\limits_\Gamma \euler^{t\lambda}(\lambda - b(\xi))^{-1} \diff \lambda .
					\end{equation*}
					the above calculations imply that we may differentiate under the integral sign and obtain $T^{(\alpha)}_t = \partial^\alpha T_t$. Thus \eqref{eq:s_b1} follows. To deduce \eqref{eq: s_a} from \eqref{eq:s_b1}, we merely need to observe that by \eqref{eq:ba} and the Leibniz rule, we obtain that there exists a constant $C_\gamma > 0$ such that if we set $K_1= C_\gamma K_0$, we obtain
			\begin{align*}
				\bigl\lVert\partial^\alpha & S_t(\xi) \bigr\rVert \\ 
				&\leq 
				C e^{\omega t}\sum_{\beta \leq \alpha} \bigl\lvert \partial^{\beta}(\euler^{-\gamma t\abs{\xi}^m})\bigr\rvert\lVert \partial^{\alpha - \beta}T_t(\xi) \rVert \\
				&\leq CK_0\euler^{\omega t - \gamma t \abs{\xi}^m} \sum_{\beta \leq \alpha}(1 + (\gamma t\abs{\xi}^{(m-1)\abs{\beta}})(t(1 +\abs{\xi})^{m-1} + t^\abs{\alpha}(1+\abs{\xi})^{(m-1)\abs{\alpha}}) \\
				&\leq CK_1\euler^{\omega t - \gamma t \abs{\xi}^m / 2} .
			\end{align*}
			By the triangle inequality we have that $N_\alpha(b) \leq C_{\gamma,\omega}N_\alpha(a)$. Thus, we obtain the statement of the lemma with $\mu = \gamma/2$ and
			\begin{align*}
			K_\alpha &= C_{\varphi,\gamma,\omega,d,m} M^{\abs{\alpha}+1}N_\alpha(a) .  \qedhere
			\end{align*}
\end{proof}
	\begin{remark}\label{rem:pertubation}
	By inspecting the proof of Lemma~\ref{lem:symbol_est}, in particular the estimate \eqref{eq:QQ}, we note that the constant $M_\alpha$ appearing in \eqref{eq: s_a} may be chosen such that it depends only on the parameters appearing in \eqref{eq:ell_constants} and
	\begin{equation*}
		\max_{\abs{\alpha} \leq m} \lVert a_\alpha \rVert ,
	\end{equation*}
	where $a_\alpha \in \L(X)$ are the coefficients of $a$. From this, we see that the estimate \eqref{eq: s_a} is stable under certain perturbations. Let for example $(A_\tau)_{\tau \in [0,1]}$ be a family of differential operators such that their symbols $(a_\tau)_{\tau \in [0,1]}$ take the form
	\begin{equation*}
		a_\tau(\xi) = a_m(\xi) + \sum_{\abs{\alpha} < m} a_{\alpha,\tau}\xi^{\alpha}, \quad (\xi \in \R^d, \tau \in [0,1]),
	\end{equation*}
	where $a_m(\xi)$ is homogeneous of degree $m$ and satisfies the normal ellipticity condition and there exists a constant  $K$ such that
	\begin{equation*}
		\lVert a_{\alpha,\tau} \rVert \leq K, \quad (\abs{\alpha} \leq m, \tau \in [0,1]) .
	\end{equation*}
	Applying the perturbation argument of Lemma~\ref{prop:pertubation} we see that there exist $\varphi, \gamma, \omega, M > 0$ independent of $\tau$ such that
	\begin{equation*}
		\lVert (\lambda + a_\tau(\xi))^{-1} \rVert \leq \frac{M}{\abs{\xi}^m + \abs{\lambda + \gamma\abs{\xi}^m}} \quad (\xi \in \R^d,~ \lambda \in \Sigma_{\varphi,-\gamma\abs{\xi}^m + \omega}) .
	\end{equation*}
	Let $(S_{t,\tau}(\xi))_{t \geq 0}$ be the semigroup generated by $-a_\tau(\xi)$. Under these conditions, it follows that for each multi-index $\alpha$ there exists a constant $M_\alpha$ independent of $\tau$ such that
	\begin{equation*}
		\lVert \partial^\alpha S_{t,\tau}(\xi) \rVert \leq M_\alpha \euler^{\omega t - \mu\abs{\xi}^m t} .
	\end{equation*}
\end{remark}
\begin{lemma}\label{lem:schwartz}
	Let $A$ be a normally elliptic differential operator with symbol $a$, denote for each $\xi \in \R^d$ the semigroup generated by $-a(\xi)$ by $(S_t(\xi))_{t \geq 0}$, and let $f \in \Schw(\R^d;X)$. For all $t \geq 0$ we define $S_tf: \R^d \to X, \xi \mapsto S_t(\xi)f(\xi)$. Then we have $S_tf \in \Schw(\R^d;X)$ and
	\begin{equation}\label{eq:c0}
		(S_tf -f) \to 0, \quad 
	\end{equation}
	and
	\begin{equation}\label{eq: s_a2}
		\frac{1}{t}\left(S_tf - f\right) \to -af
	\end{equation}
	in the topology of $\Schw(\R^d;X)$ as $t \to 0$.
\end{lemma}
\begin{proof}
	To show \eqref{eq:c0}, we need to prove that for all multi-indices $\alpha$ and $\beta$ we have
	\begin{equation*}
	\sup_{\xi \in \R^d}	\lVert \xi^\beta \partial^\alpha\left(S_t(\xi)f(\xi) - f(\xi) \right) \rVert \to 0 \quad (t \to 0) .
	\end{equation*}
	Using the Leibniz rule, it is easy to see that we need to show that for each multi-index $\alpha$ and $t \geq 0$, there exists $\Phi_\alpha(t) \geq 0$ and $N_\alpha > 0$ such that $\Phi_\alpha(t) \to 0$ as $t \to 0$ and
	\begin{equation*}
	\lVert \partial^\alpha(S_t(\xi) - 1) \rVert \leq \Phi_\alpha(t)(1 + \abs{\xi})^{N_{\alpha}}, \quad (\xi \in \R^d) .
	\end{equation*}
	In fact, by another application of the Leibniz rule, we may reduce matters to proving
	\begin{equation*}
	\lVert \partial^\alpha(T_t(\xi) - 1) \rVert \leq \Phi_\alpha(t)(1 + \abs{\xi})^{N_{\alpha}}, \quad (\xi \in \R^d) ,
	\end{equation*}
	where $T_t(\xi)$ is (as in the proof of Lemma~\ref{lem:symbol_est}) the semigroup generated by $b(\xi) = -a(\xi) + \gamma \lvert \xi \rvert^m - \omega$ with $\gamma$ as in Proposition~\ref{prop:pertubation}. 
	Suppose that $B$ is a sectorial operator on $X$ and $(V_t)_{t \geq 0}$ the associated semigroup. Then we have
	\begin{equation*}
		V_t - 1 = B \int\limits_0^t V_\tau \diff \tau, \quad (t \geq 0) .
	\end{equation*}
	Applying this with $B = b(\xi)$ where $\xi \in \R^d$, we obtain by the Leibniz rule and \eqref{eq: s_a} that there exist $C > 0$ and $C_\alpha > 0$ such that
	\begin{multline*}
		\lVert \partial^\alpha(T_t(\xi) - 1) \rVert \leq C\sum_{\beta \leq \alpha} \lVert \partial^{\alpha - \beta}b(\xi) \rVert \int\limits_0^t \lVert T^{(\beta)}_\tau \rVert \diff \tau \\ \leq C_\alpha(1 + \abs{\xi})^m\int\limits_0^t \diff \tau \leq C_\alpha t(1 + \abs{\xi})^m .
	\end{multline*}
	To show \eqref{eq: s_a2}, we need to prove that for all multi-indices $\alpha$ and $\beta$ we have
	\begin{equation*}
	\sup_{\xi \in \R^d}	\left\lVert \xi^\beta \partial^\alpha \left[ \frac{1}{t}\left(S_t(\xi)f(\xi) - f(\xi) \right) + a(\xi)f(\xi)\right]\right\rVert \to 0  
	\end{equation*}
	as $t$ tends to zero. Again, we may reduce matters to proving that for each multi-index $\alpha$ there exist $\Phi_\alpha(t) \geq 0$ and $N_\alpha > 0$ such that for all $\xi \in \R^d$ we have
	\begin{equation}\label{eq:vt}
	\left\lVert\frac{1}{t}\partial^\alpha(T_t(\xi) - 1) - \partial^\alpha b(\xi)\right\rVert \leq \Phi_\alpha(t)(1 + \abs{\xi})^{N_{\alpha}} .
	\end{equation}
	Since
	\begin{equation*}
		\frac{1}{t}(T_t(\xi) - 1) - b(\xi) = b(\xi)\frac{1}{t}\int\limits_0^t (T_\tau(\xi) - 1) \diff \tau ,
	\end{equation*}
	by the mean value theorem for integrals, we have that
	\begin{align}
		\biggl\lVert\frac{1}{t}\int\limits_0^t T_\tau(\xi) - 1 \diff \tau \biggr\rVert &\leq \sup\limits_{0 \leq s \leq t} \lVert T_s(\xi) - 1 \rVert \leq \sup\limits_{0 \leq s \leq t} \biggl\lVert b(\xi)\int\limits_0^s T_\tau(\xi) \diff \tau \biggr\rVert 
		\leq Ct \lVert b(\xi) \rVert \label{eq:mean} .
	\end{align}
	Therefore, we obtain
	\begin{equation*}
		\left\lVert\frac{1}{t}(T_t(\xi) - 1) - b(\xi)\right\rVert \leq Ct \lVert b(\xi) \rVert^2 \leq Ct(1 + \abs{\xi})^{2m} .
	\end{equation*}
	This proves \eqref{eq:vt} in the case that $\alpha = 0$. If $\alpha > 0$, we may write
	\begin{multline*}
	\frac{1}{t}\partial^\alpha(T_t(\xi) - 1) - \partial^\alpha b(\xi) \\ = (\partial^\alpha b)(\xi)\frac{1}{t}\int\limits_0^t (T_\tau(\xi) - 1) \diff \tau + \sum_{\beta < \alpha} \genfrac(){0pt}{}{\alpha}{\beta} (\partial^{\alpha - \beta}b)(\xi)\frac{1}{t}\int\limits_0^t T^{(\beta)}_\tau(\xi) \diff \tau .
	\end{multline*}
	We obtain from \eqref{eq:mean} that
	\begin{equation*}
	\biggl\lVert (\partial^\alpha b)(\xi)\frac{1}{t}\int\limits_0^t (T_\tau(\xi) - 1) \diff \tau \biggr\rVert \leq C_\alpha t\lVert \partial^\alpha b(\xi) \rVert \lVert b(\xi) \rVert \leq C_\alpha t(1 + \abs{\xi})^{2m - \abs{\alpha}}
	\end{equation*}
	If $0 \leq t \leq 1$, then it follows from \eqref{eq:valpha} that
	\begin{equation*}
	\lVert T^{(\beta)}_\tau(\xi) \rVert \leq C_\beta t(1 + \abs{\xi})^{(m - 1)\abs{\beta}}
	\end{equation*}
	which shows that
	\begin{align*}
	\biggl\lVert(\partial^{\alpha - \beta}b)(\xi)\frac{1}{t}\int\limits_0^t T^{(\beta)}_\tau(\xi) \diff \tau \biggr\rVert &\leq C_\beta(1 + \abs{\xi})^{(m-1)\abs{\beta}}\lVert \partial^{\alpha - \beta}b(\xi)\rVert\int\limits_0^t \diff \tau \\ &\leq C_{\alpha,\beta}t(1 + \abs{\xi})^{(m-1)\abs{\beta}}(1 + \abs{\xi})^{\abs{\beta}} ,
	\end{align*}
	where we have used in the second line that $m - \abs{\alpha - \beta} = m - m + \abs{\beta} = \abs{\beta}$. Summing up, we obtain
	\begin{equation*}
	\left\lVert \frac{1}{t}\partial^\alpha(T_t(\xi) - 1) - \partial^\alpha b(\xi)\right\rVert \leq C_\alpha t(1 + \abs{\xi})^{2m - \abs{\alpha}}
	\end{equation*}
	which concludes the proof.
	\end{proof}
		Let $A:\Schw^\prime(\R^d;X) \to \Schw^\prime(\R^d;X)$ be a normally elliptic differential operator with symbol $a$ and for each $\xi \in \R^d$, denote by $(S_t(\xi))_{t \geq 0}$ the semigroup generated by $-a(\xi)$ and by $S_t : \R^d \to \L(X)$ the mapping $\xi \mapsto S_t (\xi)$. As a consequence of \eqref{eq: s_a}, we obtain that for all $t \geq 0$ we have that $S_t \in \Schw(\R^d;\L(X)) \subseteq \mathcal{O}_M(\R^d;\L(X))$. Therefore, the Fourier multiplier
		\begin{equation*}
			V_t = S_t(D):\Schw'(\R^d;X) \to \Schw'(\R^d;X), \quad f \mapsto \mathcal{F}^{-1}S_t\mathcal{F}f
		\end{equation*}
		is well defined. Let $1 \leq p \leq \infty$. From Lemma~\ref{lem:mult} with $m = S_t$ we obtain that there exist constants $K$ and $\omega$ such that
		\begin{equation*}
			\lVert V_t \rVert_{L^p(\R^d;X) \to L^p(\R^d;X)} \leq Ke^{t\omega} ,
		\end{equation*}
		and by checking on elementary tensors, we see that the semigroup property
		\begin{equation*}
			 V_tV_s = V_{t + s}, \quad (s,t \geq 0)
		\end{equation*}
		holds. Thus, $\smash{(V_t^{(p)})_{t \geq 0} = (V_t|_{L^p(\R^d;X)})_{t \geq 0}}$ is a bounded semigroup. 
		If $p < \infty$, then it follows from the density of $\Schw(\R^d;X)$ in $L^p(\R^d;X)$ and the first statement of Lemma~\ref{lem:schwartz} that $\smash{V_t^{(p)}}$ is a $C_0$-semigroup.
		We denote the negative of the generator of $\smash{V_t^{(p)}}$ by $\tilde A_p$.
		\begin{lemma}
		We have $A_p = \tilde A_p$. In particular, $-A_p$ generates a semigroup given by $S_t(D)|_{L^p(\R^d)}$.
		\end{lemma}
		\begin{proof}
		Let us start by showing the inclusion $\tilde A_p \subseteq A_p$. Using the second statement of Lemma~\ref{lem:schwartz} we have
		\begin{equation*}
			\frac{1}{t}(V^{(p)}_tf -f) \to -Af
		\end{equation*}
		in the topology of $\Schw(\R^d;X)$ as $t \to 0$, and thus $\tilde A_pf = Af = A_pf$ for $f \in \Schw(\R^d;X)$. Moreover, $\Schw(\R^d;X)$ is dense in $\dom(\tilde A_p)$ since $\Schw(\R^d;X)$ is dense in $L^p(\R^d;X)$ and $\Schw(\R^d;X)$ is invariant under $\smash{V_t^{(p)}}$. Hence, using the notation $\mathcal{X}^p = L^p(\R^d;X)$, we conclude
		\begin{align*}
			\mathrm{Graph}(\tilde A_p) &= \overline { \{(f, \tilde A_pf): f \in \Schw(\R^d;X) \} }^{\mathcal{X}^p \times \mathcal{X}^p} \\ &= \overline { \{(f,A_pf): f \in \Schw(\R^d;X) \} }^{\mathcal{X}^p \times \mathcal{X}^p} 
			 \subseteq \overline { \mathrm{Graph} (A_p) }^{\mathcal{X}^p \times \mathcal{X}^p}.
		\end{align*}
		Since the embedding $J:L^p(\R^d;X) \hookrightarrow \Schw^\prime(\R^d;X)$ is continuous and $\mathrm{Graph}(A)$ is closed, $\mathrm{Graph}(A_p) = (J \times J)^{-1}\mathrm{Graph}(A)$ is closed.
		\par
		Now, observe that it follows directly from Lemma~2.5.5 in \cite{HytoenenNVW-16} that
		\begin{equation*}
			\mathrm{Graph} (A_p) = \overline { \{(f,Af): f \in \D(\R^d;X) \} }^{\mathcal{X}^p \times \mathcal{X}^p}.
		\end{equation*}
		Note that in \cite{HytoenenNVW-16}, it is assumed that the coefficients of $A$ are scalar. However, the proof given there generalizes to operator coefficients without change. Since
		\begin{equation*}
			\overline { \{(f,Af): f \in \D(\R^d;X) \} }^{\mathcal{X}^p \times \mathcal{X}^p} \subseteq \overline { \{(f,Af): f \in \Schw(\R^d;X) \} }^{\mathcal{X}^p \times \mathcal{X}^p} = \mathrm{Graph}(A_p) ,
		\end{equation*}
		we obtain $\tilde A_p = A_p$.
		\end{proof}
\subsection{Observability estimate}
Let $m \in \N$ and $A:\Schw^\prime(\R^d;X) \to \Schw^\prime(\R^d;X)$ is a normally elliptic differential operator of order $m$ with symbol $a \in \mathcal{P}_m(\R^d;X)$. Set
\begin{equation*}
	S_t:\R^d \to \L(X), \quad \xi \mapsto S_t(\xi)
\end{equation*}
where $(S_t(\xi))_{t \geq 0}$ denotes the analytic semigroup generated by $-a(\xi)$. Furthermore, for $t \geq 0$ we define $V_t = S_t(D):\Schw^\prime(\R^d;X) \to \Schw^\prime(\R^d;X)$ the Fourier multiplier with symbol $S_t$. Let $p \in [1,\infty]$. Then the restriction $\smash{(V_t^{(p)})_{t \geq 0} = (V_t|_{L^p(\R^d;X)})_{t \geq 0}}$ is a bounded semigroup on $L^p (\R^d ; X)$. If $p < \infty$, the semigroup $\smash{(V_t^{(p)})_{t \geq 0}}$ is strongly continuous and we denote its generator by $A_p$. In the following, we will write $\smash{V_t = V_t^{(p)}}$ when there's no risk of confusion.
\begin{theorem}\label{thm:obs}
Let $\rho,T > 0$, $L \in (0,\infty)^d$, $E \subseteq \R^d$ a $(\rho,L)$-thick set, and $1 \leq p,r \leq \infty$. Then there exists a constant $C_{\mathrm{obs}} > 0$ such that for all $f \in L^p(\R^d;X)$ it holds that
\begin{equation*}
	\lVert V_Tf \rVert_{L^p(\R^d;X)} \leq C_{\mathrm{obs}}\lVert V_{(\cdot)}f \rVert_{L^r([0,T];L^p(E;X))} .
\end{equation*}
\end{theorem}
We choose a function $\varphi \in C_c^\infty(\R)$ such that $0 \leq \varphi \leq 1$, $\spt \varphi \subseteq \ball{0}{1}$ and $\varphi = 1$ on $\ball{0}{1/2}$. For $\xi \in \R^d$ we set $\chi_\lambda(\xi) = \varphi(\abs{\xi} / \lambda)$ and define $P_\lambda = \chi_\lambda(D)$.
\begin{lemma}\label{lem:diss}
There exist constants $c_1,c_2, \lambda_0 > 0$ depending only on $a$ such that for all $t \geq 0$ and $\lambda \geq \lambda_0$ we have
\begin{equation*}
		\lVert \F^{-1}(1 - \chi_\lambda)S_t \rVert_{L^1(\R^d;\L(X))} \leq c_1 e^{-c_2t\lambda^m} .
\end{equation*}
Moreover, for all $p \in [1, \infty]$, $t \geq 0$ and $\lambda \geq \lambda_0$ we have
\begin{equation*}
	\lVert (I - P_\lambda)V_t \rVert_{L^p(\R^d;X) \to L^p(\R^d;X)} \leq c_1 e^{-c_2t\lambda^m} .
\end{equation*}
\end{lemma}
\begin{proof}
	We consider 3 separate cases. 
	\par
	{Case 1:} $t > 1$. Let $\varepsilon > 0$. By Lemma \ref{lem:mult}, it suffices to show that
	\begin{equation}\label{eq:fmobs}
		\lVert (1 - \chi_\lambda)S_t \rVert_{W^{d+1,\infty}} + \max\limits_{\abs{\alpha} \leq d+1} \sup\limits_{\xi \in \R^d} \abs{\xi}^{\abs{\alpha} + \varepsilon}\lVert \partial^\alpha ((1 - \chi_\lambda)(\xi)S_t(\xi)) \rVert \leq c_1 e^{-c_2\lambda^m t} .
	\end{equation}
	for some constants $c_1, c_2$. For this, we observe that by the Leibniz rule, for each multi-index $\alpha$, there exists a constant $C_\alpha$ such that
	\begin{equation*}
		\lVert \partial^\alpha ((1 - \chi_\lambda(\xi))S_t(\xi)) \rVert \leq C_\alpha\sum\limits_{\beta \leq \alpha} \abs{\partial^\beta(1 - \chi_\lambda)(\xi)}\lVert \partial^{\alpha - \beta}S_{t}(\xi) \rVert .
	\end{equation*}
	Observe that if $\lambda \geq 1$, there exists an absolute constant $C_\beta > 0$ such that
	\begin{equation*}
		\abs{\partial^\beta(1 - \chi_\lambda)(\xi)} \leq C_\beta\mathbf{1}_{\abs{\xi} \geq \lambda/2} .
	\end{equation*}
	Therefore, by \eqref{eq: s_a}, it follows that there exist $K_\alpha,\omega,\mu$ such that
	\begin{equation*}
		\lVert \partial^\alpha (1 - \chi_\lambda(\xi))S_t(\xi) \rVert \leq K_\alpha \mathbf{1}_{\abs{\xi} \geq \lambda/2}\euler^{\omega t - \mu \abs{\xi}^m t} .
	\end{equation*}
	Choosing $\lambda_0^m = \max\{1,2^{m+1}\mu^{-1}\omega\}$, we obtain for all multi-indices $\alpha$ such that $\abs{\alpha} \leq d+1$ and $\lambda \geq \lambda_0$
	\begin{equation*}
		\lVert \partial^\alpha ((1 - \chi_\lambda(\xi))S_t(\xi)) \rVert  \leq K_\alpha\euler^{-\mu 2^{-m-1}\lambda^m t} .
	\end{equation*}
	This shows that there exist constants, $c_1^\prime,c_2^\prime$ such that
	\begin{equation*}
		\lVert (1 - \chi_\lambda)S_t \rVert_{W^{d+1,\infty}} \leq c_1'\euler^{-c_2'\lambda^m t} .
	\end{equation*}
	Moreover, observe that
	\begin{equation*}
	\abs{\xi}^{\abs{\alpha} + \varepsilon}\lVert \partial^\alpha ((1 - \chi_\lambda)(\xi)S_t(\xi))\rVert \leq K_\alpha \abs{\xi}^{\abs{\alpha} + \varepsilon}\mathbf{1}_{\abs{\xi} \geq \lambda/2}\euler^{\omega t - \mu \abs{\xi}^m t} .
	\end{equation*}
	and thus, employing that $t > 1$, it follows that there exists $K_\alpha^\prime$ such that
	\begin{equation*}
	\abs{\xi}^{\abs{\alpha} + \varepsilon}\lVert \partial^\alpha ((1 - \chi_\lambda)(\xi)S_t(\xi))\rVert \leq K_\alpha^\prime \mathbf{1}_{\abs{\xi} \geq \lambda/2}\euler^{\omega t - (\mu/2) \abs{\xi}^m t} .
	\end{equation*}
	Arguing as before, we find $c_1^{\prime \prime},c_2^{\prime \prime}$ such that
	\begin{equation*}
		\abs{\xi}^{\abs{\alpha} + \varepsilon}\lVert \partial^\alpha (1 - \chi_\lambda(\xi))S_t(\xi) \rVert \leq c_1^{\prime \prime}\euler^{-c_2^{\prime \prime}\lambda^m t} .
	\end{equation*}
	We now obtain \eqref{eq:fmobs} by summing up.
	\par
	{Case 2:} $0 \leq t \leq 1,$ $t^{1/m}\lambda > 1$.
	We begin with two easy observations. Firstly, if $m:\R^d \to \L(X)$ is such that $\lVert \F^{-1}m \rVert_{L^1(\R^d;\L(X))} < \infty$, then for any $\mu > 0$, we have
	\begin{equation}\label{eq:rescale}
		\lVert \F^{-1}[m(\mu\cdot)] \rVert_{L^1(\R^d)} = \mu^{-d}\lVert (\F^{-1}m)(\mu^{-1}\cdot) \rVert_{L^1(\R^d)} = \lVert \F^{-1}m \rVert_{L^1(\R^d)} .
	\end{equation}
	Secondly, if $(W_t)_{t \geq 0}$ is a $C_0$-semigroup with generator $B$, then for any $\mu > 0$ the rescaled semigroup defined by $(\tilde W_t)_{t \geq 0} = (W_{\mu t})_{t \geq 0}$ is associated to $\mu B$.
	Denote by $(T_{\tau,t})_{\tau \geq 0}$ the semigroup on $X$ associated to $-ta(t^{-1/m}\cdot) \in \L(X)$. We consider the rescaled symbol
	\begin{equation*}
		\sigma_{t,\lambda} = ((1 - \chi_\lambda)S_t)(t^{-1/m}\cdot) = (1 - \chi_{t^{1/m}\lambda})S_t(t^{-1/m}\cdot) = (1 - \chi_{t^{1/m}\lambda})T_{1,t} .
	\end{equation*}
	It follows from \eqref{eq:rescale} that it suffices to show that there exist constants $c_1,c_2 > 0$ such that
	\begin{equation*}
		\lVert \F^{-1}\sigma_{t,\lambda} \rVert \leq c_1e^{-c_2t\lambda^m}
	\end{equation*}
	Observe that
	\begin{equation*}
		ta(t^{-1/m}\xi) = a_m(\xi) + \sum_{\abs \alpha < m}t^{1 - \frac{\abs{\alpha}}{m}}a_\alpha\xi^\alpha ,
	\end{equation*}
	and therefore $N_0(a_m(\xi) - ta(t^{-1/m}\xi)) \leq K$ for some constant $K$ independent of $t$.
	It thus follows from Lemma~\ref{lem:symbol_est} and Remark \ref{rem:pertubation} that for each multi-index $\alpha$ there exist constants $K_\alpha, \mu > 0$ such that
	\begin{equation*}
		\lVert \partial^\alpha T_{1,t}(\xi) \rVert \leq K_\alpha e^{-\mu\abs{\xi}^m} .
	\end{equation*}
	Moreover, since $t^{1/m}\lambda > 1$, we have that for each multi-index $\beta$ there exist constants $C_\beta > 0$ such that
	\begin{equation*}
		\abs{\partial^\beta(1 - \chi_{t^{1/m}\lambda})(\xi)} \leq C_\beta\mathbf{1}_{\abs{\xi} \geq t^{1/m}\lambda/2} \label{eq:T1t} .
	\end{equation*}
	By the Leibniz rule, it therefore follows that there exist constants $C_\alpha > 0$ such that
	\begin{equation*}
		\lVert \partial^\alpha \sigma_{t,\lambda}(\xi) \rVert \leq C_\alpha\mathbf{1}_{\abs{\xi} \geq t^{1/m}\lambda/2}e^{\omega t-\mu\abs{\xi}^m} .
	\end{equation*}
	Let $\varepsilon > 0$. Arguing as in Case 1, we see that there exist $\lambda_0 > 0$ and constants $c_1',c_2'$ and $c_1'', c_2''$ such that for all $\lambda \geq \lambda_0$
	\begin{equation*}
		\lVert \partial^\alpha \sigma_{t,\lambda}(\xi) \rVert \leq c_1'e^{-c_2't\lambda^m}
	\end{equation*}
	and
	\begin{equation*}
		\abs{\xi}^{m + \varepsilon}\lVert\partial^\alpha \sigma_{t,\lambda} \rVert \leq c_1''e^{-c_2''t\lambda^m}
	\end{equation*}
	for all multi-indices $\alpha$ with $\abs{\alpha} \leq d+1$. We can thus apply Lemma~\ref{lem:mult} also in this case. 
	\par
	{Case 3}: $0 \leq t \leq 1,$ $ 0 \leq t^{1/m}\lambda \leq 1$.
	Employing the notation of Case 2, we see from \eqref{eq:T1t} and Lemma~\ref{lem:mult} that there exists $A > 0$ such that
	\begin{equation*}
		\lVert \F^{-1}T_{1,t} \rVert \leq A .
	\end{equation*}
	Again by \eqref{eq:rescale} it follows that there exists $B > 0$ such that
	\begin{equation*}
		\lVert \F^{-1}(1 - \chi_{t^{1/m}\lambda}) \rVert \leq B .
	\end{equation*}
	It thus follows from Young's inequality that
	\begin{equation*}
		\lVert \mathcal{F}^{-1}\sigma_{t,\lambda} \rVert \leq AB .
	\end{equation*}
	Since we have due to the restriction $0 \leq t^{1/m}\lambda \leq 1$ for any $c > 0$ that
	\begin{equation*}
		AB \leq ABe^ce^{-ct\lambda^m}
	\end{equation*}
	the result also follows in this case.
\end{proof}
\begin{proof}[Proof of Theorem~\ref{thm:obs}]
	We apply Theorem~A.1 from \cite{BombachGST-23} to the semigroup $\smash{(V_t^{(p)})_{t \geq 0}}$ acting on the Banach space $L^p(\R^d;X)$ and the family of quasi-projections $(P_\lambda)_{\lambda > 0}$. We only need to verify that there exist positive constants $\lambda_0$, $d_0,d_1,d_2,d_3$ such that for all $f\in L^p(\R^d;X)$, all $\lambda > \lambda_0$ and all $t \in [0,T/2]$ we have
	\begin{align*} 
			 \lVert P_\lambda f \rVert_{L^p(\R^d;X)} &\le d_0 e^{d_1 \lambda} \lVert \mathbf{1}_EP_\lambda f \rVert_{L^p(\R^d;X)}  \\
			 \intertext{and}
			 \lVert (I-P_\lambda) V_t f \rVert_{L^p(\R^d;X)} &\le d_2 e^{-d_3 \lambda^{m} t} \lVert f \rVert_{L^p(\R^d;X)}  ,
	\end{align*}
	 and that the mapping $\Phi: [0,T] \ni t \mapsto \lVert \Eins_E V_tf \rVert_{L^p(\R^d;E)}$ is measurable. The first inequality is satisfied by Theorem~\ref{thm:LS}, whereas the second inequality follows from Lemma~\ref{lem:diss}. Measurability of $\Phi$ follows from the strong continuity of $V_t$ if $p < \infty$. Suppose now that $p = \infty$. By Proposition 1.3.1 of \cite{HytoenenNVW-16}, we have that the linear subspace
	\[
	\left\{ f \mapsto \int\limits_{\R^d} \dual{g(x),f(x)}_{X^\prime \times X}\diff x \colon g \in L^1(\R^d;X^\prime) \right\} \subseteq (L^\infty(\R^d;X))^{\prime}
	\]
	is norming for $L^\infty(\R^d;X)$, meaning that
	\[
	\Phi(t) = \lVert \Eins_E V_t f \rVert_{L^\infty(\R^d;X)} = \sup  \left\{ \int\limits_{\R^d} \dual{g(x),\Eins_E V_tf(x)}_{X^\prime \times X}\diff x \colon \lVert g \rVert_{L^1(\R^d;X^\prime)} = 1 \right\}.
	\]
	By the strong continuity of $V_t$, the map
	\[
	t \mapsto \int\limits_{\R^d} \dual{g(x),\Eins_E V_tf(x)}_{X^\prime \times X}\diff x
	\]
	is continuous for each $g \in L^1(\R^d;X^\prime)$. Thus, $\Phi$ is lower semicontinuous as it is the supremum of continuous functions and therefore measurable.
\end{proof}
\subsection{Null-controllability}
Let $E \subseteq \R^d$ be measurable, $p \in  [1,\infty)$ and $T > 0$. Set $\mathcal{X}^p = L^p(\R^d;X)$ and consider the controlled system
\begin{equation}\label{eq:control}
\partial_ty(t) + A_p y(t) = \mathbf{1}_Eu(t), \quad y(0) = y_0 \in \mathcal{X}^p, \quad t \in [0,T] .
\end{equation}
Let $r \in [1,\infty]$. Given a control function $u \in L^r([0,T];\mathcal{X}^p)$, the mild solution of \eqref{eq:control} is given by
\begin{equation*} 
 y(t) = V_ty_0 + \int\limits_0^t V_{t-s}\mathbf{1}_Eu(s) \diff s .
\end{equation*}
We say that the system \eqref{eq:control} is null-controllable in $L^r([0,T];\mathcal{X}^p)$ in time $T$ if for any $y_0 \in \mathcal{X}^p$ there exists an $u \in L^r([0,T];\mathcal{X}^p)$ such that $y(T) = 0$. Setting
\begin{equation*}
 \mathcal{B}_T:L^r([0,T];\mathcal{X}^p) \to \mathcal{X}^p, \quad u \mapsto \int\limits_0^T V_{t-s}\mathbf{1}_Eu(s) \diff s ,
 \end{equation*}
 we see that \eqref{eq:control} is null-controllable in $L^r([0,T];\mathcal{X}^p)$ at time $T$ if and only if $\ran (V_T) \subseteq \ran (\mathcal{B}_T)$. Moreover, we define $\eqref{eq:control}$ to be \emph{approximately null-controllable} at time $T$ if   $\ran (V_T) \subseteq \overline{\ran (\mathcal{B}_T)}$ with the bar denoting the norm closure of the set $\ran (\mathcal{B}_T)$ in $\mathcal{X}^p$. Thus, $\eqref{eq:control}$ is \emph{approximately null-controllable} at time $T$ if and only if for all $\varepsilon > 0$ and all $y_0 \in \mathcal{X}^p$ there exists $u \in L^r([0,T];\mathcal{X}^p)$ such that $\lVert y(T) \rVert_{\mathcal{X}^p}  < \varepsilon$.
 \begin{theorem} \label{thm:control}
 Let $\rho > 0$, $L \in (0,\infty)^d$ and $E$ $(\rho , L)$-thick, and assume that $X^\prime$ has the Radon-Nikodym property. Then,
 \begin{enumerate}[(a)]
 \item if $p \in (1,\infty)$, the system \eqref{eq:control} is null-controllable in $L^r([0,T];\mathcal{X}^p)$ at time $T$.
 \item if $p = 1$, the system \eqref{eq:control} is approximately null-controllable in $L^r([0,T];\mathcal{X}^p)$ at time $T$.
  \end{enumerate}
 \end{theorem}
 \begin{proof}
 Let $q$ be such that $p^{-1} + q^{-1} = 1$ and $s$ such that $r^{-1} + s^{-1} = 1$. Write $\mathcal{Y}^q = (\mathcal{X}^p)^\prime.$ It holds that $\mathcal{Y}^q = L^{q}(\R^d;X^\prime)$ due to the Radon-Nikodym property of $X^\prime$. For $t \geq 0$ we set $W_t = V_t^\prime$. By Douglas' Lemma, the statement of the theorem is equivalent to the fact that there exists a constant $C_{\mathrm{obs}}$ such for every $f \in \mathcal{Y}^q$ we have the observability estimate
 \begin{equation}\label{eq:obs_dual}
 		\lVert W_Tf \rVert_{\mathcal{Y}^q} \leq C_{\mathrm{obs}}\lVert \mathcal{B}_T^\prime f \rVert_{L^r([0,T];\mathcal{X}^p)^\prime}\,.
 \end{equation}
 By \cite[Theorem~2.1]{Vieru-05} it holds
 \[
 \lVert \mathcal{B}_T^ \prime f \rVert_{L^r([0,T];\mathcal{X}^p)^\prime}
 = \lVert W_{(\cdot)} f \rVert_{L^{s}([0,T];\mathcal{Y}^q)}
 \]
Recall that $S_t$ is the symbol of $V_t$. To obtain \eqref{eq:obs_dual}, we note that due to Proposition~\ref{adjoint-multiplier} and Lemma~\ref{lem:diss} we obtain for all $\lambda \geq \lambda_0$ the dissipation estimate
 \begin{equation*}
 	\lVert (I - P_\lambda)W_t \rVert \leq \lVert \mathcal{F}^{-1}S_t(1 - \chi_\lambda)\rVert_{L^1(\R^d;\L(X))} \leq c_1 e^{-c_2t\lambda^m} .
 \end{equation*}
 Since the uncertainty principle also holds for functions with values in $X^\prime$, we obtain the observability estimate as in the proof of Theorem~\ref{thm:obs}.
 \end{proof}
\appendix
\section{Proof of Theorem~\ref{thm:LS}}
First we assume $L = (1,1,\ldots,1)$,  and fix $\lambda \in (0,\infty)^d$, $\rho > 0$, a $(\rho,1)$-thick set $E$, and $f \in L^p (\R^d ; X)$ with $\spt \F f \subseteq \Pi_\lambda$ as in the assumptions of the theorem. Note that $f$ is analytic since $\spt \F f$ is compact, see Section~\ref{sec:preliminaries}.
For $k \in \Z^d$ we denote by $\Lambda_k = (-1/2 , 1/2)^d + k \subseteq \R^d$ the open unit cube centered at $k$. Let 
\begin{equation}\label{eq:A}
A > \frac{1}{1-(2^d + 1)^{-1/d}} \in (3/2,2) ,
\end{equation}
and $C_2 > 0$ be the absolute constant from Proposition~\ref{prop:Bernstein}. We call $k \in \Z^d$ \emph{bad} if there exists $\alpha \in \N_0^d$ with $\alpha \not = 0$ such that
\[
 \lVert \Eins_{\Lambda_k} \partial^\alpha f\rVert_{L^p (\R^d ; X)} \geq 2^{d} A^{\lvert \alpha \rvert } (C_2 \lambda)^{\alpha } \lVert \Eins_{\Lambda_k} f \rVert_{L^p (\R^d ; X)} .
\]
Otherwise we call $k \in \Z^d$ \emph{good}. Moreover, we will use the notation 
\[
 \Lambda_{\text{bad}} = \bigcup_{ \genfrac{}{}{0pt}{}{k \in \Z^d \colon}{k\ \text{is bad}} }
 \Lambda_k 
 \quad \text{and} \quad \Lambda_{\text{good}} = \bigcup_{ \genfrac{}{}{0pt}{}{k \in \Z^d \colon}{k\ \text{is good}} } \Lambda_k  .
\]
\begin{lemma} 
 \begin{enumerate}[(i)]
  \item We have
  \[
  \lVert \Eins_{\Lambda_{\mathrm{good}}} f \rVert_{L^p (\R^d ; X)}
  \geq 
  C_3 \lVert f \rVert_{L^p (\R^d ; X)} , 
  \]
  where
  \[
  C_3 := C_3 (A) := 1 - \left( \frac{1}{2^d} \left[ \left( \frac{1}{1-1/A} \right)^d - 1 \right]\right)^{1/p} \in (0,1)
  \]
  if $p \in [1,\infty)$, and $C_3 = 1$ if $p = \infty$   .
  \item\label{lemma:1b} There exists $B > A$ such that for all good $k \in \Z^d$ there exists $x \in \Lambda_k$ such that for all $\alpha \in \N_0^d$ we have
\begin{equation*}
 \lVert \partial^\alpha f (x) \rVert_X \leq 4^d B^{\lvert \alpha \rvert } (C_2 \lambda)^{\alpha  }
 \lVert \Eins_{\Lambda_k} f \rVert_{L^p (\R^d ; X)} .
 \end{equation*}
 \end{enumerate}
 \label{lemma:1}
\end{lemma}
\begin{proof}
It follows by definition that for all $p \in [1,\infty)$
\begin{align*}
 \lVert \Eins_{\Lambda_{\mathrm{bad}}} f \rVert_{L^p (\R^d ; X)}^p
 &= \sum_{k \in \Z^d \cap \Lambda_{\text{bad}}} \lVert \Eins_{\Lambda_{k}} f \rVert_{L^p (\R^d ; X)}^p
 \leq 
 \sum_{k \in \Z^d \cap \Lambda_{\text{bad}}} \sum_{\genfrac{}{}{0pt}{}{\alpha \in \N_0^d \colon}{\alpha \not = 0} } \frac{\lVert \Eins_{\Lambda_k} \partial^\alpha f \rVert_{L^p (\R^d ; X)}^p}{2^{dp} A^{p \lvert \alpha \rvert } (C_2 \lambda)^{p \alpha}} \\
 &=
 \sum_{\genfrac{}{}{0pt}{}{\alpha \in \N_0^d \colon}{\alpha \not = 0}} \frac{\lVert \Eins_{\Lambda_{\text{bad}}} \partial^\alpha f \rVert_{L^p (\R^d ; X)}^p}{2^{dp} A^{p \lvert \alpha \rvert} (C_2 \lambda)^{p \alpha }}
 \leq
 \sum_{\genfrac{}{}{0pt}{}{\alpha \in \N_0^d \colon}{\alpha \not = 0}} \frac{\lVert \partial^\alpha f \rVert_{L^p (\R^d ; X)}^p}{2^{dp} A^{p \lvert \alpha \rvert} (C_2 \lambda)^{p \alpha }} .
\end{align*}
By Proposition~\ref{prop:Bernstein}, and since $A \geq 1$ we conclude for all $p \in [1,\infty)$ that 
\begin{align*}
 \lVert \Eins_{\Lambda_{\mathrm{bad}}} f \rVert_{L^p (\R^d ; X)}^p
 &\leq
 \sum_{\genfrac{}{}{0pt}{}{\alpha \in \N_0^d \colon}{\alpha \not = 0}} \frac{\lVert f \rVert_{L^p (\R^d ; X)}^p}{2^{dp} A^{p \lvert \alpha \rvert }} 
 =
 \frac{1}{2^{dp}} \left[ \left( \frac{1}{1-1/A^p} \right)^d - 1 \right]\lVert f \rVert_{L^p (\R^d ; X)}^p .
 \\
 &\leq \frac{1}{2^d} \left[ \left( \frac{1}{1-1/A} \right)^d - 1 \right]\lVert f \rVert_{L^p (\R^d ; X)}^p = (1-C_3)^p \lVert f \rVert_{L^p (\R^d ; X)}^p .
\end{align*}
For $p \in [1,\infty)$ it follows that
\[
 \lVert \Eins_{\Lambda_{\mathrm{good}}} f \rVert_{L^p (\R^d ; X)}
 \geq 
 C_3 \lVert f \rVert_{L^p (\R^d ; x)} .
\]
By \eqref{eq:A} we have $C_3 \in (0,1)$. {This proves the first claim in the case $p \in [1,\infty)$. If $p = \infty$ the proof is even easier. By the definition of bad and Proposition~\ref{prop:Bernstein} we have
\begin{align*}
\lVert \Eins_{\Lambda_{\mathrm{bad}}} f \rVert_{L^\infty (\R^d ; X)} 
&\leq 
\sup\limits_{k \in \Z^d \colon k \ \text{bad}} \sum_{\genfrac{}{}{0pt}{}{\alpha \in \N_0^d \colon}{\alpha \not = 0}} \frac{\lVert \Eins_{\Lambda_k} \partial^\alpha f \rVert_{L^\infty (\R^d ; X)}}{2^d A^{\lvert \alpha \rvert} (C_2 \lambda)^\alpha}  \\
&\leq
\frac{1}{2^d} \left[\left(\frac{1}{1-1/A} \right)^d - 1 \right]\lVert  f \rVert_{L^\infty (\R^d ; X)}  .
\end{align*}
Since the prefactor in the last inequality is strictly smaller than one, we conclude that $\lVert \Eins_{\Lambda_{\mathrm{good}}} f \rVert_{L^\infty (\R^d ; X)} = \lVert  f \rVert_{L^\infty (\R^d ; X)}$.}
\par
In order to prove part (\ref{lemma:1b}) we consider the contraposition, that is, for all $B > A$ there exists a good $k \in \Z^d$ such that for all $x \in \Lambda_k$ there is $\alpha \in \N_0^d$ with
 \[
 \lVert \partial^\alpha f (x) \rVert_X > 4^d B^{\lvert \alpha \rvert } (C_2 \lambda)^{\alpha }
 \lVert \Eins_{\Lambda_k} f \rVert_{L^p (\R^d ; X)} .
\]
This and the definition of good implies that there exists a good $k \in \Z^d$ such that we have 
\begin{align*}
 2^{d} \lVert \Eins_{\Lambda_k} f \rVert_{L^p (\R^d ; X)}
  <
  \sum_{\alpha \in \N_0^d}\frac{\lVert \Eins_{\Lambda_k} \partial^\alpha f \rVert_{L^p (\R^d ; X)}}{2^{d} B^{ \lvert \alpha \rvert } (C_2 \lambda)^{\alpha}} 
 &  \leq \sum_{\alpha \in \N_0^d} \left( \frac{A}{B} \right)^{\lvert \alpha \rvert}
  \lVert \Eins_{\Lambda_k} f \rVert_{L^p (\R^d ; X)} .
\end{align*}
Choosing, for instance, $B = 3A$ we obtain
\[
 2^{d} \lVert \Eins_{\Lambda_k} f \rVert_{L^p (\R^d ; X)}
 \leq
 \left(\frac{1}{1-\left( 1/3 \right)}\right)^d  \lVert \Eins_{\Lambda_k} f \rVert_{L^p (\R^d ; X)} ,
\]
a contradiction.
\end{proof}
Let $s = 1$ if $p \in [1,\infty)$ or some arbitrary number $s \in (0, 1)$ if $p = \infty$, $k \in \Z^d$ be good and $y \in \Lambda_{k}$ be such that $\lVert f(y) \rVert_X \geq s \lVert \Eins_{\Lambda_{k}} f \rVert_{L^p (\R^d ; X)}$. Furthermore, let $\Omega \subseteq \Lambda_{k}$ be a measurable set to be chosen later. Then, using spherical coordinates, we have
\[
 \lvert \Omega \rvert
 =
 \int_{\Lambda_{k}} \Eins_{\Omega} (x) \diff x
 =
 \int_{S^{d-1}} \int_{r=0}^{r (\vartheta)} \Eins_{\Omega} (y+ r \vartheta) r^{d-1} \diff r \diff\sigma (\vartheta) ,
\]
where $r (\vartheta) = \sup \{t > 0 \colon y + t \vartheta \subseteq \Lambda_{k}\}$, and where $\sigma$ denotes the surface measure.
There exists a $\vartheta_0 \in S^{d-1}$ such that 
\begin{equation} \label{eq:line-segment}
 \lvert \Omega \rvert \leq {\sigma( S^{d-1} )} \int_{0}^{r(\vartheta_0)} \Eins_{\Omega} (y+ r \vartheta_0) r^{d-1} \diff r .
\end{equation}
Indeed, if the converse inequality to \eqref{eq:line-segment} would hold for all $\vartheta \in S^{d-1}$, then averaging over $S^{d-1}$ would give a contradiction.
Let now $I_0 = \{ y + r \vartheta_0 \colon r>0 , \ y+r\vartheta_0 \in \Lambda_{k}\}$ be the largest line segment in $\Lambda_{k}$ starting in $y$ in the direction of $\vartheta_0$. Since $r (\vartheta_0) \leq d^{1/2}$ we conclude from \eqref{eq:line-segment} that $\lvert \Omega \rvert \leq \sigma(S^{d-1}) d^{(d-1)/2} \lvert \Omega  \cap I_0 \rvert$, where, with some abuse of notation, we use the notation
$
 \lvert \Omega  \cap I_0 \rvert = \int_{0}^{r(\vartheta_0)} \Eins_{\Omega} (y+ r \vartheta_0) \diff r 
$. 
\par
Now we define the function $F \colon \C^d \to X$ by 
\[
 F (w) = \frac{1}{N} (\mathcal{L} \mathcal{F} f) (y+w \lvert I_0 \rvert \vartheta_0) ,
\]
where $\mathcal{L}$ denotes the inverse Fourier-Laplace transform, cf.\ Section~\ref{sec:preliminaries}, and where $N$ de\-notes the normalization $N = s\lVert \Eins_{\Lambda_{k}} f \rVert_{L^p (\R^d ; X)}$. Note that $F$ is an entire function which extends $(1/N)f(y + \cdot \abs{I_0}\vartheta_0)$ to $\C^d$, see Section~\ref{sec:preliminaries}. Thus we have for all $w \in \C^d$ and $x \in \R^d$
\begin{align*}
 \lVert F (w) \rVert_X & \leq 
 \frac{1}{N}\sum_{\alpha \in \N_0^d} \frac{\lVert f^{(\alpha)} (x) \rVert_X}{\alpha !} \prod_{i = 1}^d \lvert (y + w \lvert I_0 \rvert \vartheta_0 - x)_i \rvert^{\alpha_i} 
\end{align*}
By Lemma~\ref{lemma:1} there exists $x_0 \in \Lambda_{k}$ such that for all $w \in \C^d$
\begin{equation*}
 \lVert F (w) \rVert_X \leq\frac{4^{d}}{N}\sum_{\alpha \in \N_0^d} \frac{ B^{\lvert \alpha \rvert} (C_2\lambda)^{\alpha } \lVert \Eins_{\Lambda_{k}} f \rVert_{L^p (\R^d ; X)}}{\alpha !} \prod_{i = 1}^d \lvert (y + w \lvert I_0 \rvert \vartheta_0 - x_0)_i \rvert^{\alpha_i} .
\end{equation*}
Since for all $w \in \discc{5}$ we have 
\[
y - x_0 + w \lvert I_0 \rvert \vartheta_0 \in \bigtimes\limits_{i=1}^d \discc{6 \sqrt{d}},
\]
we conclude for all $w \in \discc{5}$ 
\begin{align*}
 \lVert F (w) \rVert_X 
 &\leq 
 \frac{4^{d}}{N } \sum_{\alpha \in \N_0^d} \frac{B^{\lvert \alpha \rvert} (C_2\lambda)^{ \alpha} \lVert \Eins_{\Lambda_{k}}f \rVert_{L^p (\R^d ; X)}}{\alpha !} (6 \sqrt{d})^{\lvert \alpha \rvert} \\
 & = 
 \frac{4^{d}}{N} \lVert \Eins_{\Lambda_{k}} f \rVert_{L^p (\R^d ; X)} \exp\left( 6 d^{1/2} B C_2 \lvert \lambda \rvert \right) 
 =
 4^{d} \exp\left( 6 d^{1/2}  B C_2 \lvert \lambda \rvert \right) =:M .
\end{align*}
We recall that by assumption on $y$ we have $\lVert F (0) \rVert_X = N^{-1} \lVert f(y) \rVert_X \geq 1$. By Proposition~\ref{prop:funktionentheorie} we have for all closed intervals $I \subseteq \R$ with $0 \in I$ and $\lvert I \rvert = 1$, and all measurable $A \subseteq I$ that
 \begin{equation*}
	\sup_{x \in A} \lVert F(x) \rVert_X \geq \left( \frac{\abs{A}}{C_1} \right)^{\frac{\ln(M)}{\ln(2)}} \sup_{x \in I} \lVert F(x) \rVert_X.
\end{equation*}
with some absolute constant $C_1 \geq 1$. Choose $I = [0,1]$ and $A = \{t \in [0,1] \colon y + t \vartheta_0 \in \Omega \cap I_0\}$, then
 \begin{equation*}
  \sup_{x \in \Omega \cap I_0} \lVert f(x) \rVert_X \geq \left( \frac{\abs{\Omega \cap I_0}}{C_1} \right)^{\frac{\ln(M)}{\ln(2)}} \sup_{x \in I_0} \lVert f(x) \rVert_X.
\end{equation*}
By our choice of $y$ we have that $\sup_{x \in I_0} \lVert f(x) \rVert_X \geq \lVert f (y) \rVert_X \geq s \lVert \Eins_{\Lambda_{k}} f \rVert_{L^p (\R^d ; X)}$. Moreover, we have shown above that $\lvert \Omega \rvert \leq \sigma(S^{d-1}) d^{(d-1)/2} \lvert \Omega \cap I_0 \rvert$. Hence, we conclude
\begin{equation*}
\sup_{x \in \Omega } \lVert f(x) \rVert_X \geq
 \left( \frac{\abs{\Omega}}{C_1 \sigma(S^{d-1}) d^{(d-1)/2}} \right)^{\frac{\ln(M)}{\ln(2)}} s \lVert \Eins_{\Lambda_{k}} f \rVert_{L^p (\R^d ; X)}  .
\end{equation*}
Recall that $s = 1$ if $p \in [1,\infty)$, and that the above inequality holds for arbitrary $s \in (0,1)$ if $p = \infty$. By taking limits we obtain
\begin{equation} \label{eq:supremum}
\sup_{x \in \Omega } \lVert f(x)\rVert_X \geq
 \left( \frac{\abs{\Omega}}{C_1 \sigma(S^{d-1}) d^{(d-1)/2}} \right)^{\frac{\ln(M)}{\ln(2)}}  \lVert \Eins_{\Lambda_{k}} f \rVert_{L^p (\R^d ; X)}  .
\end{equation}
Now we choose
\[
 \Omega = \left\{x \in \Lambda_{k}\colon \left( \frac{\abs{E \cap \Lambda_{k}}}{2C_1 \sigma(S^{d-1}) d^{(d-1)/2}} \right)^{\frac{\ln(M)}{\ln(2)}} \lVert \Eins_{\Lambda_{k}} f \rVert_{L^p (\R^d ; X)} > \lVert f(x) \rVert_X 
 \right\} .
\]
By Ineq.~\eqref{eq:supremum} and the definition of $\Omega$, we obtain
\begin{align*}
\left( \frac{\abs{E \cap \Lambda_{k}}}{2 \lvert \Omega \rvert} \right)^{\frac{\ln(M)}{\ln(2)}} 
 \sup_{x \in \Omega} \lVert f(x) \rVert_X 
 &\geq 
 \left( \frac{\abs{E \cap \Lambda_{k}}}{2 C_1 \sigma ( S^{d-1} ) d^{(d-1)/2}} \right)^{\frac{\ln(M)}{\ln(2)}} \lVert \Eins_{\Lambda_{k}} f \rVert_{L^p (\R^d ; X)}  \\
 &\geq
 \sup_{x \in \Omega} \lVert f(x) \rVert ,
\end{align*}
and thus $\lvert \Omega \rvert \leq \lvert E \cap \Lambda_{k} \rvert / 2$. The definition of $\Omega$ implies 
\begin{align*}
 \lVert \Eins_{E \cap \Lambda_{k}} f \rVert_{L^p (\R^d ; X)}
 &\geq 
 \lVert \Eins_{E \cap \Lambda_{k}} \Eins_{\Omega^{\mathrm{c}}} f \rVert_{L^p (\R^d ; X)} \\
 &\geq 
 \left( \frac{\abs{E \cap \Lambda_{k}}}{2C_1 \lvert S^{d-1} \rvert d^{(d-1)/2}} \right)^{\frac{\ln(M)}{\ln(2)}} \lVert \Eins_{\Lambda_{k}} f \rVert_{L^p (\R^d ; X)} \lVert \Eins_{E \cap \Lambda_{k} \cap \Omega^{\mathrm{c}}} \rVert_{L^p (\R^d)} .
\end{align*}
Moreover, since $\lvert \Omega \rvert \leq \lvert E \cap \Lambda_{k} \rvert / 2$, we have 
\begin{equation*}
 \lvert E \cap \Lambda_{k} \cap \Omega^{\mathrm{c}} \rvert 
 = \lvert E \cap \Lambda_{k} \rvert - \lvert E \cap \Lambda_{k} \cap \Omega \rvert  
 \geq 
 \lvert E \cap \Lambda_{k} \rvert - \lvert \Omega \rvert
 \geq 
 \frac{\lvert E \cap \Lambda_{k} \rvert}{2} .
\end{equation*}
Since $E$ is thick, we have that $1 \geq \lvert E \cap \Lambda_k \rvert > 0$ , thus $E \cap \Lambda_k \cap \Omega^{\mathrm{c}}$ has positive measure as well. We conclude
\begin{align*}
 \lVert \Eins_{E \cap \Lambda_{k} \cap \Omega^{\mathrm{c}}} \rVert_{L^p (\R^d)}
 \geq 
 \frac{\lvert E \cap \Lambda_{k} \rvert}{2} 
\end{align*}
Hence, using $C_4 := 2C_1 \lvert S^{d-1} \rvert d^{(d-1)/2} \geq 2$, the fact that $\lvert E \cap \Lambda_{k} \rvert \geq \rho$ by the definition of the thick set $E$, and $\rho / C_4 \leq 1$ we can conclude
\begin{align*}
\lVert \Eins_{E \cap \Lambda_{k}} f \rVert_{L^p (\R^d ; X)}
 &\geq
 \left( \frac{\abs{E \cap \Lambda_{k}}}{2C_1 \lvert S^{d-1} \rvert d^{(d-1)/2}} \right)^{\frac{\ln(M)}{\ln(2)}} \lVert \Eins_{\Lambda_{k}} f \rVert_{L^p (\R^d ; X)} 
 \left( \frac{\lvert E \cap \Lambda_{k} \rvert}{2} \right) 	\\
 &
 \geq 
 \left( \frac{\rho}{C_4} \right)^{\frac{\ln(M)}{\ln(2)}+1} \lVert \Eins_{\Lambda_{k}} f \rVert_{L^p (\R^d ; X)} .
\end{align*}
Since $k \in \Z^d$ was arbitrary but good, we can either sum over all good cubes (if $p \in [1,\infty)$), or take the supremum over all good cubes (if $p = \infty$), and obtain by using Lemma~\ref{lemma:1} 
\[
\lVert \Eins_{E} f \rVert_{L^p (\R^d ; X)}
\geq \lVert \Eins_{E \cap \Lambda_{\text{good}}} f \rVert_{L^p (\R^d ; X)} \geq C_3 \left( \frac{\rho}{C_4} \right)^{\frac{\ln(M)}{\ln(2)} + 1} \lVert f \rVert_{L^p (\R^d ; X)} .
\]
By the definitions of $M$, $C_3$ and $C_4$ and using that $\rho \leq 1$, we find that there exists a constant $C_d \geq 1$ depending only on the dimension $d$ such that for all $p \in [1,\infty]$ we have
\[
\lVert \Eins_{E} f \rVert_{L^p (\R^d ; X)}
\geq \left( \frac{\rho}{C_d} \right)^{C_d(1 + \abs{\lambda}_1) } \lVert f \rVert_{L^p (\R^d ; X)} .
\]
This proves the statement in the case $L = (1,1,\ldots,1)$. Let now $L \in [0,\infty)^d$ be arbitrary. Theorem~\ref{thm:LS} follows by applying the result for $L = (1,1,\ldots,1)$ to the function $f \circ T_L$ where $T_L:\R^d \to \R^d$ is given by $T_L x = (L_kx_k)_{k=1}^{d}$. 
\paragraph{Acknowledgments.} Both authors thank Thomas Kalmes and Christian Seifert for stimulating discussions, which significantly helped to improve the manuscript.

\end{document}